\documentclass[11pt]{amsart}

\usepackage{amsmath}
\usepackage{amssymb}
\usepackage{graphicx}
\usepackage{enumerate}
\usepackage{url}
\usepackage{dsfont} % needed for double-struck 1
\newtheorem{theorem}{Theorem}[section]
\newtheorem{corollary}[theorem]{Corollary}
\newtheorem{lemma}[theorem]{Lemma}
\newtheorem{proposition}[theorem]{Proposition}
\theoremstyle{definition}
\newtheorem{definition}[theorem]{Definition}
\newtheorem{example}[theorem]{Example}
\newtheorem{remark}[theorem]{Remark}

\numberwithin{equation}{section}

\newtheorem{obs}[theorem]{Observation}

\def\mtx#1{\begin{bmatrix} #1 \end{bmatrix}}

\def\nJ#1{\widetilde J_{#1}}
\newcommand{\bo}{\ensuremath{\mathds{1}}}
\def\nbo#1{\widetilde \bo_{#1}}

\newcommand{\R}{\mathbb{R}}

\newcommand{\sym}{\mathcal{S}}

\newcommand{\bx}{{\bf x}}

\newcommand{\bzero}{{\bf 0}}

\newcommand{\Rnn}{\R^{n\times n}}

\newcommand{\vect}{\operatorname{vec}}
\newcommand{\qM}{q_M}
\newcommand{\mult}{\operatorname{mult}}
\newcommand{\mr}{\operatorname{mr}}
\newcommand{\M}{\operatorname{M}}
\newcommand{\oml}{{\bf m}}

\newcommand{\spec}{\operatorname{spec}}

\newcommand{\nul}{\operatorname{null}} 
\newcommand{\supp}{\operatorname{supp}}

\newcommand{\rank}{\operatorname{rank}}
\newcommand{\dev}{\operatorname{dev}}

\newcommand{\tr}{\operatorname{tr}}
\newcommand{\diag}{\operatorname{diag}}
\newcommand{\dist}{\operatorname{dist}}

\newcommand{\lam}{\lambda}

\newcommand{\x}{\times}
\newcommand{\wh}{\widehat}
\newcommand{\wt}{\widetilde}
\newcommand{\bit}{\begin{itemize}}
\newcommand{\eit}{\end{itemize}}
\newcommand{\ben}{\begin{enumerate}}
\newcommand{\een}{\end{enumerate}}
\newcommand{\beq}{\begin{equation}}
\newcommand{\eeq}{\end{equation}}
\newcommand{\bea}{\begin{eqnarray*}} % * means no number
\newcommand{\eea}{\end{eqnarray*}}
\newcommand{\bpf}{\begin{proof}}
\newcommand{\epf}{\end{proof}\ms}
\newcommand{\bmt}{\begin{bmatrix}}
\newcommand{\emt}{\end{bmatrix}}
\newcommand{\ms}{\medskip}

\newcommand{\cp}{\, \Box\,}
\newcommand{\lc}{\left\lceil}
\newcommand{\rc}{\right\rceil}
\newcommand{\lf}{\left\lfloor}
\newcommand{\rf}{\right\rfloor}
\newcommand{\du}{\,\dot{\cup}\,}
\newcommand{\noi}{\noindent}

\title[Minimum number distinct graph eigenvalues]{Applications of analysis to the determination of the minimum number of distinct eigenvalues of a graph}

\author[B. Bjorkman]{Beth Bjorkman}
\address[B. Bjorkman]{Department of Mathematics, Iowa State University, Ames, IA 50011, USA}
\email{{\tt bjorkman@iastate.edu}}

\author[L. Hogben]{Leslie Hogben}
\address[L. Hogben]{Department of Mathematics, Iowa State University, Ames, IA 50011, USA and American Institute of Mathematics, 600 E. Brokaw Road, San Jose, CA 95112, USA}
\email{{\tt hogben@aimath.org}}

\author[S. Ponce]{Scarlitte Ponce}
\address[S. Ponce]{Department of Mathematics, Iowa State University, Ames, IA 50011, USA}
\email{{\tt sponce@iastate.edu}}

\author[C. Reinhart]{Carolyn Reinhart}
\address[C. Reinhart]{Department of Mathematics, Iowa State University, Ames, IA 50011, USA}
\email{{\tt reinh196@iastate.edu}}

\author[T. Tranel]{Theodore Tranel}
\address[T. Tranel]{Department of Mathematics, Iowa State University, Ames, IA 50011, USA}
\email{{\tt  teddyt@iastate.edu}}

\date{\today}



\begin{document}
\maketitle
%\linenumbers

%\vspace{-15pt} 
\begin{abstract} 
We establish new  bounds on  the minimum number of distinct eigenvalues among real symmetric matrices with nonzero off-diagonal pattern described by the edges of a graph  and apply these to  determine the minimum number of distinct eigenvalues of several families of graphs and small graphs.  \end{abstract}

%%%%%%%%%%%%%%%%%%%%%%%%%%%%%%%%%%%%%%%%%%%

\section{Introduction}\label{sintro} 
 Inverse eigenvalue problems appear in various contexts throughout mathematics and engineering, and refer to determining all possible lists of eigenvalues (spectra)  for matrices fitting some description. The {\em inverse eigenvalue problem of a graph (IEPG)} refers to determining the possible spectra of real symmetric matrices whose pattern of nonzero off-diagonal entries is described by the edges of a given graph.  Graphs often describe relationships in a physical system and the eigenvalues of associated matrices govern the behavior of the system.   The  IEPG and related variants have been  of interest for many years.  Various parameters have been used to study this problem, most importantly the maximum multiplicity of an eigenvalue of a matrix described by the graph   (see, for example, \cite{HLA2}).  In \cite{AACFMN13} the authors introduce the parameter $q(G)$ as the minimum number of eigenvalues among the matrices described by the graph.  In this paper we establish additional techniques for bounding $q$ and determine its value for various families of graphs.

The Strong Multiplicity Property (SMP) and the Strong Spectral Property (SSP) are recently developed tools that were introduced in \cite{BFHHLS16} (see also Section \ref{sSMP}) and have enabled significant progress on the IEPG. 
 The SMP and SSP have their roots in the implicit function theorem.  The SMP allows us  to perturb along the intersection of the pattern  manifold and the fixed ordered multiplicity list manifold  (along the fixed spectrum manifold for SSP) under suitable conditions.  In this paper we apply the SMP and SSP and additional matrix tools such as the Kronecker product of matrices (see Section \ref{sprod}) to establish bounds on the minimum number of distinct eigenvalues of a graph.  We then apply these results to  determine the minimum number of distinct eigenvalues of several families of graphs and of small graphs.

 In this paper, a  {\em graph} is a pair $(V(G),E(G))$ where $V(G)=\{1,2,\ldots ,n\}$ and $E(G)$ is a set of 2-element subsets of $V(G)$, each having the form $\{u,v\}$ where $1\leq u<v\leq n$. We also denote an edge $\{u,v\}$ as $uv$; in this case vertices $u$ and $v$ are {\em adjacent} and are {\em neighbors}.  The {\em neighborhood} of vertex $u$ is $N(u)=\{v\in V(G) : uv\in E(G)\}$. A {\em leaf} is a vertex with only one neighbor.  The {\em order} of $G$ is the number of vertices, $|V(G)|$.  A graph $G'=(V',E')$ is a subgraph of $G=(V,E)$ if $V'\subseteq V$ and $E' \subseteq E$.  We say that a subgraph $\wt G=(\wt V, \wt E)$ is  a spanning subgraph of $G$ if $\wt V= V$. 

Let $G$ be a graph of order $n$.  The set $\sym(G)$ of  matrices representing $G$   is the set of real symmetric $n\x n$ matrices $A=[a_{ij}]$ such that for $i\not= j$, $a_{ij} \neq 0$ if and only if $ij\in E(G)$ (the diagonal is unrestricted). For a  matrix $A$, the number of distinct eigenvalues of $A$ is denoted $q(A)$ and the minimum number of distinct eigenvalues of a graph $G$ is
\[  q(G)=\min\{q(A) : A\in\mathcal{S}(G)\}.\]

Section \ref{sSMP} contains a discussion of the SMP and SSP and applies them to the determination of $q$.  Section \ref{sprod} presents bounds on $q(G)$ for graphs $G$ constructed as Cartesian, tensor, or strong products.  Section \ref{sotherops} presents results about $q(G)$ for certain types of block-clique graphs and joins.  The ability  of these graph operations to raise or lower  $q$ is discussed in Section \ref{graphops}.  
We determine values of $q(G)$ for all graphs of order 6 in Section \ref{s6} and then summarize the values of all graphs $G$ for which $q$ is currently known in Section \ref{sqfam}.
  The remainder of this introduction contains additional definitions and  results from the literature that will be used.

%-----

\subsection{Terminology and notation}

Matrices discussed are real and symmetric, so all eigenvalues are real and each matrix has an orthonormal basis of eigenvectors. Let $A$ be an $n\x n$ matrix.  The {\em spectrum} of $A$ is the multiset of eigenvalues  of $A$ (repeated according to multiplicity) and is denoted by $\spec(A)$. The notation $\lam_k(A)$ denotes the $k$th eigenvalue of $A$ with $\lam_1(A)\le \dots\le \lam_n(A)$. 
If the matrix $A$ has distinct eigenvalues $\mu_1 < \mu_2<\ldots<\mu_q$ with multiplicities $m_1,m_2,\dots,m_q$, respectively, then
the {\em ordered multiplicity list} of $A$ is $\oml(A)=(m_1,m_2,\dots,m_q)$.
 In this paper we denote the set of distinct eigenvalues of a matrix $A$ by $\dev(A)$. 
A {\em principal submatrix} of $A$ is a submatrix obtained from $A$ by deleting a set of rows and the corresponding set of columns.     For $1\le k\le n$, the principal submatrix  $A(k)$ is the $(n-1)\x(n-1)$ matrix obtained from $A$ by deleting row $k$ and column $k$ from $A$.
The formal definitions of the maximum nullity (which is equal to the maximum multiplicity of an eigenvalue) and minimum rank are 
  \[\M(G)=\max\{  \nul(A): A\in \sym(G)\}  \mbox{ and } \mr(G)=\mbox{min}\{ \rank (A): A\in \mathcal{S}(G) \}.\]\vspace{-10pt}
  It is easy to observe that $\mr(G)+\M(G)=|V(G)|$, so the study of maximum nullity is equivalent to the study of minimum rank.

   A {\em path} of order $n$ is a graph $P_n$
with   $V(P_n) = \{ v_i: 1 \le i \le n \}$ and   $E(P_n) = \{ \{  v_i,  v_{i+1} \}:1 \le i \le n-1 \}$.  The {\em length} of $P_n$ is the number of edges, i.e., $n-1$. A graph $G$ is {\em connected} if for every pair of distinct vertices $u$ and $v$, $G$ contains  a path  from $u$ to $v$.  In a connected graph $G$, the {\em distance} from $u$ to $v$, denoted by $\dist_G(u,v)$, is the minimum length of a path from $u$ to $v$.
If $n\geq 3$, a {\em cycle} of order $n$ is a graph $C_n$ with  $V(C_n) = \{ v_i:
1 \leq i \leq n \}$ and  $E(C_n) = \left\{ \left\{  v_i,  v_{i+1} \right\}: 1 \le i \le n-1 \right\}\cup\left\{\{v_n,v_1\}\right\}$.  A {\em complete graph} of order $n$ is a graph $K_n$
with  $V(K_n) = \{ v_i: 1 \le i \le n \}$ and $E(K_n) = \{ \{  v_i,  v_j \}: 1 \le i<j \le n \}$.
A {\em complete bipartite graph} with partite sets $X$ and $Y$ of orders $s$ and $t$ is the graph $K_{s,t}$
with  $V(K_{s,t}) = X\cup Y$ where $X= \{ x_i: 1 \le i \le s \}$ and $Y=\{ y_i: 1 \le i \le t \} $ are disjoint,  and $E(K_{s,t}) = \{ \{  x_i,  y_j \}: 1 \le i \le s, 1 \le j \le t \}$.

 % {\red [add  defs $A(k)$ (and $A[S]$ if needed)]}

%-----

\subsection{Results cited} 

\begin{theorem}[Interlacing Theorem]\label{thm:interlacing} {\rm \cite[Theorem 8.10]{Zhang}}  For $A\in\Rnn$ and $1\le k \le n$,
\[\lam_1(A)\le \lam_1(A(k))\le \lam_2(A)%\le \lam_2(A(k))
\le \dots\le \lam_{n-1}(A)\le \lam_{n-1}(A(k))\le \lam_n(A).\]
More generally, if $B$ is a principal submatrix of $A$ obtained by deleting the rows and columns corresponding to a set of $m$ indices, then 
\[\lam_k(A)\le \lam_k(B)\le \lam_{k+m}(A)\mbox{ for }k=1,\dots,n-m.\]
\end{theorem}

\begin{proposition}\label{mr+1}{\rm \cite[Proposition 2.5]{AACFMN13}} For a graph $G$, $q(G)\leq \mr(G)+1$.\end{proposition}

\begin{obs}\label{M}{\rm \cite[p. 23]{BFHHLS16}} % p 23
  For a graph $G$ on $n$ vertices, $\lceil\frac{n}{\M(G)}\rceil\leq q(G)$.\end{obs}

%\begin{theorem}\label{qpapercartprodthm} {\rm \cite[Theorem 6.7]{AACFMN13}}  Let $G$ be a graph on $n$ vertices.  Then $q(G\cp K_2)\le 2q(G)-2$.\end{theorem}

\begin{obs}\label{q2} {\rm \cite[p. 678]{AACFMN13}}  If $q(G)=2$, then there exists a symmetric orthogonal matrix $A\in\sym(G)$.
\end{obs}

\begin{proposition}\label{qpaperKn}   For any $n\geq 2$, $q(K_n)=2$ {\rm \cite[Lemma 2.2]{AACFMN13}}. For any $n\geq 1$, $q(P_n)=n$ {\rm \cite[p. 676]{AACFMN13}}. For any $n\geq 3$, $q(C_n)=\lc \frac n 2 \rc$ {\rm \cite[Lemma 2.7]{AACFMN13}}.

\end{proposition}

\begin{theorem}\label{qpaperKmn} {\rm \cite[Corollary 6.5]{AACFMN13}} 
For any $m,n$ with $1\leq m\leq n$, \[ q(K_{m,n}) = \left\{ \begin{array}{cc} 2 & m=n \\ 3 & m<n \end{array} \right.\!\!. \]
\end{theorem}

\begin{theorem}\label{qjoinsamesize} {\rm \cite[Theorem 5.2]{MS16}}  Let $G$ and $G'$ be connected graphs of order $n$.  Then $q(G\vee G')=2$ and there is a matrix $M\in\sym(G\vee G')$ with $\oml(M)=(n,n)$.
\end{theorem}

\noi The next theorem is often referred to as the ``unique shortest path theorem."  
\begin{theorem}\label{quniquedistance} {\rm \cite[Theorem 3.2]{AACFMN13}}  If there are vertices $u$ and $v$ in a connected graph $G$ such that $\dist_G(u,v)=d$  and the path of length $d$ is unique, then $q(G)\geq d+1$.
\end{theorem}

\begin{theorem}\label{independentsets} {\rm \cite[Theorem 4.4]{AACFMN13}} For a connected graph $G$ on $n$ vertices, if $q(G)=2$, then for any independent set of vertices $\{v_1,\ldots,v_k\}$  we have
  {\small \[ \left|\bigcup_{i\neq j} (N(v_i)\cap N(v_j))\right|\geq k   \text{ or }   \left|\bigcup_{i\neq j} (N(v_i)\cap N(v_j))\right|=0.\]}
\end{theorem}

\begin{theorem}\label{thm:qisGminus1} {\rm \cite[Theorem 51]{BFHHLS16}}  
{A graph $G$ has $q(G)\geq |V(G)|-1$ if and only if}  $G$ is one of the following:
 a path;
  the disjoint union of a path and an isolated vertex;
  {a path with one leaf attached to an interior vertex}; 
 {a path  with an extra edge joining two vertices  at distance $2$}. 
\end{theorem}

A path on $n$ vertices with one leaf attached to an interior vertex is called a {\em generalized star} and is denoted by $S(k-1,n-k-1,1)$, where $k$ is the vertex with the extra leaf with path vertices numbered in path order.  An order $n$ path  with an extra edge joining the two vertices $k+1$ and $k+3$ ($0\le k\le n-3$) is called a {\em generalized bull} and is denoted by $GB(k,n-k-3)$.

%%%%%%%%%%%%%%%%%%%%%%%%%%%%%%%%%%%%%%%% Gluing Results
\section{Strong properties}\label{sSMP}

The Strong Spectral Property (SSP) and Strong Multiplicity Property (SMP) were introduced in \cite{BFHHLS16}  and additional properties and applications are given  in \cite{IEPG2}.  These  properties can yield powerful results. In this section we define and apply them.

The entry-wise product of $A, B\in\Rnn$ is denoted by $A\circ B$ and the trace (sum of the diagonal entries) of $A$ is denoted by $\tr A$. An  $n \times n$ symmetric matrix $A$ satisfies the {\em Strong
Spectral Property (SSP)} \cite{BFHHLS16} provided no nonzero
symmetric matrix  $X$ satisfies
\begin{itemize}
\item $ A \circ X = 0 =I \circ X$ and
\item $AX-XA = 0$.
\end{itemize}
A $n \times n$ symmetric matrix $A$ satisfies the {\em Strong
Multiplicity Property (SMP)} \cite{BFHHLS16} provided no nonzero
symmetric matrix  $X$ satisfies
\begin{itemize}
\item $ A \circ X = 0 =I \circ X$,
\item $AX-XA = 0$, and
\item $\tr(A^i X) = 0$ for $i=0, \ldots, n-1$.
\end{itemize}
If a matrix has SSP, then it also has SMP, but not conversely  \cite{BFHHLS16}.  The definitions of the SMP and SSP just given are linear algebraic conditions that allow the application of the Implicit Function Theorem to perturb one or more pairs of zero entries to nonzero entries while maintaining the nonzero pattern of other entries and preserving the ordered multiplicity list or spectrum (see \cite{BFHHLS16} for more information).
The next theorem will be  applied to give an upper bound on  $q$.

\begin{theorem} \label{thm:subIEPG1} {\rm \cite[Theorem 20]{BFHHLS16}} Let $G$ be a graph  
and let $\wt G$ be a spanning subgraph of $G$. If $\wt A\in\sym(\wt G)$ has  SMP, then there exists $A\in\mathcal{S}(G)$ with SMP having the same multiplicity list as $\wt A$.\end{theorem}

The {\em SMP minimum number of distinct eigenvalues} of a graph $H$ is defined in  \cite{BFHHLS16} to be
\[\qM(H)=\min\{q(A) : A\in \mathcal{S}(H),A\text{ has SMP}\}.\]
%In the next result we use $\qM$ of a subgraph to provide an upper bound on $q$.
The next result is clear from the definitions and Theorem \ref{thm:subIEPG1}.
\begin{obs} \label{thm:sub} Let $G$ be a graph %with $|E(G)|\neq \emptyset$ 
and let $\wt G$ be a spanning subgraph of $G$.  Then $q(G)\leq \qM(G)\leq \qM(\wt G)$.
\end{obs}
%\bpfLet $\wt A\in \mathcal{S}(\wt G)$ with SMP such that $q(\wt A)=\qM(\wt G)$. Then  there exists $A\in\mathcal{S}(G)$ with SMP such that  $\oml(A)=\oml(\wt A)$ by Theorem \ref{thm:subIEPG1}.  Thus, $q(G)\leq \qM(G)\leq q(A)=q(\wt A)=\qM(\wt G)$. %\vspace{-10pt} \epf

A {\em Hamilton cycle} in a graph is a cycle that includes every vertex.  The next result is a simplified form of  \cite[Corollary 49]{BFHHLS16} and follows from $q_M(C_n)=\lc\frac n 2\rc$  \cite[Theorem 48]{BFHHLS16}.

\begin{corollary}\label{HamCycle}  {\rm  \cite{BFHHLS16}} Let $G$ be a graph of order $n$ that has a Hamilton cycle. Then $q(G)\le\lc\frac n 2\rc$.
\end{corollary}

It is known (see, for example, \cite{BFHHLS16}) that for any set  of distinct eigenvalues $\lam_1<\dots<\lam_n$ and any graph $G$ of order $n$ there is a matrix  $A\in\sym(G)$ with $\spec(A)=\{\lam_1,\dots,\lam_n\}$.  The next result includes the additional requirement that every entry of the diagonal of $A$ is nonzero.  

\begin{theorem} \label{nzdiag} Let  $G$ be a graph of order $n$. Then any set of $n$ distinct nonzero real numbers can be realized by some matrix $A\in\sym(G)$ that has SSP and has all  diagonal entries nonzero.
\end{theorem}
\bpf
Let $\lam_1,\dots,\lam_n$ be distinct nonzero real numbers.  As noted in \cite[Remark 15]{BFHHLS16}, there is a matrix $A\in\sym(G)$ that has SSP and $\spec(A)=\{\lam_1,\dots,\lam_n\}$.  The matrix $A$ is obtained from the matrix  $D=\diag(\lam_1,\dots,\lam_n)$ by a perturbation of the entries; note that $D$ has SSP since the diagonal entries are distinct \cite[Theorem 34]{BFHHLS16}.  Since such perturbation may be chosen arbitrarily small, we may assume the diagonal entries of $A$ are all nonzero.   %\vspace{-10pt}
\epf

The next two results about strong properties  appear in \cite{BFHHLS16} and \cite{IEPG2} and are used in Section \ref{s6}.  Theorem \ref{thm:SSP-nec-suff-conds} allows verification of the SSP or SMP for $A\in\sym(G)$  by computation of the rank of a matrix constructed from $A$ and $G$.  Lemma \ref{augment} allows us to import results from the solution of the IEPG for graphs of order 5 to determine the value of $q$ for order 6.  Some definitions are needed first. The {\em support} of a vector $\bx$ is $\supp(\bx)=\{i:x_i\ne 0\}$.
Let $H$ be a graph with vertex set $\{1,2,\ldots, n\}$ and edge-set $\{e_1, \ldots, e_p\}$.  We denote the endpoints of $e_k$ by $i_k$ and $j_k$.  
For a symmetric $n\x n$ matrix $A={[a_{ij}]}$, we denote by $\vect_H(A)$ the $p \times 1$ vector whose $k$th coordinate is $a_{i_kj_k}$.
Thus $\vect_H(A)$ makes a vector out of the elements of $A$ corresponding to the edges in $H$. 
%Note that $\vect_H( \cdot )$ defines a linear transformation from $\Rsn$ to $\R^p$. 
The matrix $E_{ij}$ denotes the $n\x n $ matrix with a $1$ in the $i,j$-position and $0$ elsewhere, and
$K_{ij}$ denotes the $n\times n$ skew-symmetric matrix  $E_{ij}-E_{ji}$.
  The {\em complement} $\overline{G}$ of $G$ is the graph with the same vertex set as $G$ and edges exactly where $G$ does not have edges. 
The next theorem is used to determine whether a matrix has SSP.

\begin{theorem}
\label{thm:SSP-nec-suff-conds}{\rm \cite[Theorem 31]{BFHHLS16}} Let $G$ be a graph,  let $A\in\sym(G)$ 
and let $p$ be the number of edges in ${\overline{G}}$. 
Then 
  $A$ has  SSP if and only if the $p\x {n\choose 2} $ matrix whose columns are   $\vect_{\overline{G}}\,(AK_{ij}-K_{ij} A)$ for $1\leq i < j \leq n$ has rank $p$.
\end{theorem}

\begin{lemma}[Augmentation Lemma]\label{augment} {\rm \cite[Lemma 7.5]{IEPG2}} Let $G$ be a graph on vertices $\{1,\ldots,n\}$ and $A\in\sym(G)$.  Suppose $A$ has  SSP and $\lam$ is an eigenvalue of $A$ with multiplicity {$k\ge 1$}. Let $\alpha$ be a subset of $\{1,\ldots, n\}$ of cardinality $k+1$ with the property that for every eigenvector $\bx$ of $A$ corresponding to $\lam$, $|\supp(\bx)\cap \alpha |\ge 2$.  %$|\{x_i\ne 0:i \in \alpha\}|\ge 2$.    
Construct $H$ from $G$ by appending vertex $n+1$ adjacent exactly to the vertices in $\alpha$. Then there exists a matrix $B\in\sym( H)$ such that $B$ has  SSP,  the multiplicity of $\lam$ has increased from $k$ to $k+1$, and other eigenvalues and their  multiplicities are unchanged from those of $A$. \end{lemma}

The Augmentation Lemma is usually applied to a specific matrix where the eigenvectors can be determined (as in Section \ref{s6}).  However, it is also possible to apply it without a specific matrix as is done in the next corollary.  

\begin{corollary}\label{addunivtx}  Suppose $G$ is a graph, each vertex of $G$ has at least two neighbors, and $H$ is constructed from $G$ by adding a new vertex adjacent to every vertex of $G$.  If $A\in\sym(G)$ has SSP and $\oml(A)=(m_1,\dots,m_r)$, then for each $j=1,\dots,r$  there exists a matrix $B_j\in\sym(H)$ such that $B_j$ has SSP, the distinct eigenvalues of $B_j$ are the same as those of $A$, and $\oml(A)=(m_1,\dots,m_{j-1},m_j+1,m_{j+1},\dots,m_r)$. \end{corollary}
\bpf  
 We apply the Augmentation Lemma with $\alpha = \{1,\dots,n\}$, so $|\alpha|\ge m_j+1$. For any vector $\bx$, $|\supp(\bx)\cap \alpha|=|\supp(\bx)|$.  Suppose $|\supp(\bx)|=1$ for some eigenvector $\bx$.  Let $k$ be the position containing the one nonzero entry of  $\bx$.  Then  $A\bx=\lam \bx$ implies the $k$th column of $A$ has at most one nonzero entry, which is impossible since $A\in\sym(G)$ and every vertex of $G$ has at least two neighbors. So $|\supp(\bx)\cap \alpha|\ge 2$. Then   there exists a matrix $B_j\in\sym(H)$ with the required properties by the Augmentation Lemma.  
\epf

\begin{corollary}\label{Kn-e}  For $n\ge 4$, $q(K_n-e)=2$ and there is a matrix $M\in\sym(K_n-e)$ with SSP and $\oml(M)=\left(\lc\frac n 2\rc, \lf\frac n 2\rf\right)$. \end{corollary}
\bpf The graphs $K_4-e$ and $K_5-e$ are done in \cite{BFHHLS16}, so assume $n\ge 6$.  For $n=2k$, the result follows from joining $K_k$ with $K_k-e$ by Theorem \ref{qjoinsamesize}, which shows there exists a matrix $A\in\sym(K_n- e)$ with $\oml(A)=(k,k)$.  We show that $A$ has SSP, and the result then follows from Corollary \ref{addunivtx}.  Note that  $A\circ X = O = I\circ X$ implies $X=[x_{ij}]$ has only one symmetrically placed pair of possibly nonzero entries, say $x_{12}=x_{21}=x$.  Then  $(AX-XA)_{23}=xa_{23}$.  Since  $a_{23}\ne 0$, $x=0$ and $X=O$. 
\epf%\vspace{-10pt}

%%%%%%%%%%%%%%%%%%%%%%%%%%%%%%%%%%%%%%%% 
\section{Graph products}\label{sprod} 

In this section we compute bounds for   $q$ for Cartesian, tensor, and strong products of graphs, and in  some cases we determine the value of $q$ for graphs constructed by these products.  The {Kronecker product}  of matrices plays a central role in constructing matrices realizing graph parameters for graphs that are products.
For $A\in\Rnn$ and $A'\in\R^{n'\x n'}$, the {\em Kronecker product} of $A$ and $A'$ is  the $nn'\x nn'$ matrix

\[{A \otimes A'}=\left[\begin{array}{cccc}
a_{11}A'&a_{12}A'&\cdots &a_{1n}A'\\
a_{21}A'&a_{22}A'&\cdots &a_{2n}A'\\
\vdots &\vdots &\ddots &\vdots\\
a_{n1}A'&a_{n2}A'&\cdots &a_{nn}A'\\
\end{array}\right]\!.\]

For sets or multisets of real numbers $S$ and $T$, we define sets or multisets $S+T=\{s+t: s\in S, t\in T\}$ and $ST=\{st : s\in S, t\in T\}$ (for sets duplicates are removed, but  for multisets duplicates are left in place). 
It is well known that  $\spec(A\otimes A)= \spec(A)\spec(A')$ (see, for example, \cite[Theorem 4.8]{Zhang}); this implies $\dev(A\otimes A)= \dev(A)\dev(A')$.  %We denote the set $\{1,\ldots,n\}$ by $[n]$.

%----------------

\subsection{Cartesian products}\label{scp}

The {Cartesian product} of graphs $G$ and $G'$, denoted by $G\cp G'$, has vertex set $V(G)\times V(G')$ and edge set $\{(u,v)(x,y) : u=x \text{ and } vy\in E(G') \text{ or }  v=y \text{ and } ux\in E(G)\}$.  We present several  bounds on the value of $q$ for Cartesian products of graphs that apply when certain hypotheses on the constituent graphs are met.

\begin{proposition}\label{cpap}  Let $G_1$ and $G_2$ be graphs.  If $q(G_i)$ can be realized by matrices $A_i\in\sym(G_i), i=1,2$ with $\dev(A_i)=\{1,2,\dots,q_i(G)\}$, then $q(G_1\cp G_2)\le q(G_1)+q(G_2)-1$.
\end{proposition}

\begin{proof}
%Let $A_1\in\sym(G_1)$ such that $q(A_1) =q(G_1)$ and $\spec(A_1) = \{ 1,2,$ $\dots,$ $q(G_1) \}$ and  $A_2\in\sym(G_2)$ such that $q(A_2) =q(G_2)$ and $\spec(A_2) = \{ 1,2,\dots,q(G_2) \}$. 
Assume the required $A_i$ exist. We observe that 
$\dev(A_1 \otimes I + I \otimes A_2)=\dev(A_1)+\dev(A_2)=\{2,\ldots, q(G_1)+q(G_2)\}$.
Therefore there are $q(G_1)+q(G_2)-1$ distinct eigenvalues of $(A_1 \otimes I + I \otimes A_2)\in \sym(G_1\cp G_2)$, and so $q(G_1\cp G_2)\leq q(G_1)+q(G_2)-1$.\end{proof}

Since any set of distinct eigenvalues can be realized as the eigenvalues of a path, we have the following result.
\begin{corollary}\label{cpapPath}  If $G$  is a graph  such that $q(G)$ can be realized by a matrix $A\in\sym(G)$ with $\dev(A)=\{1,2,\dots,q(G)\}$, then $q(G\cp P_s)\le q(G)+s-1$.
\end{corollary}

%The bound in Corollary \ref{cpapPath} is tight for $P_s\cp P_2$ by Corollary \ref{PsboxP2}
For $s=2$, the bound $q(G\cp P_2)\le 2q(G)-2$ given in \cite[Theorem 6.7]{AACFMN13} is better than that in Corollary \ref{cpapPath} when $q(G)=2$, and the bounds are equal for $q(G)=3$, but otherwise the bound in Corollary \ref{cpapPath} is better. %However, it is still too large in many cases. %; {we do not have an example of a graph where the bound in Corollary  \ref{cpapPath}  gives equality}.  %, because for $q(G)=2$ it gives $q(G\cp P_2)\le 3$ even though $q(G\cp P_2)= 2$ (Theorem \ref{qpapercartprodthm}).

\begin{corollary}\label{cpapCyc}  If $G$  is a graph  such that $q(G)$ can be realized by a matrix $A\in\sym(G)$ with $\dev(A)=\{1,2,\dots,q(G)\}$, then $q(G\cp C_{n})\le q(G)+\lc\frac n 2 \rc$.
\end{corollary}
\bpf  Assume the hypotheses.  For $C_{2k+1}$ we can realize the ordered multiplicity list (2,...,2,1) with any spectrum by \cite{F80}.  For $C_{2k}$ we can realize the ordered multiplicity list (2,...,2) with any spectrum by \cite{FF10} (cited in \cite[Lemma 2.7]{AACFMN13}). \epf

\begin{proposition}\label{cPathLowerBound} Let $G$ and $G'$ be  graphs and let $d$ denote the length of the unique shortest path between vertices of distance $d$ in $G$.  If $q(G)=d+1$, then $q(G\cp G')\ge q(G)$.\end{proposition}

\bpf Assume $q(G)=d+1$. Let $v_1,v_{d+1}\in V(G)$ such that $\dist_G(v_1,v_{d+1})=d$ and let $v_1,v_2,\dots, v_{d+1}$ be the unique shortest path of length $d$ from $v_1$ to $v_{d+1}$ in $G$. Then for any $v'\in V(G')$, $(v_1,v'),(v_2,v'),\dots, (v_{d+1},v')$ is a path of length $d$ in $G\cp H$. It is clear that $\dist_{G\cp H}((v_1,v'),(v_{d+1},v'))=d$. This path is the unique path of length $d$ since a path involving $(v_i,u')$ for some other $u'\in V(G')$ would be longer and any other path $(v_1,v'),(w_2,v'),\dots,$ $(w_d,v'), (v_{d+1},v')$ would contradict the uniqueness of the path in $G$. So by Theorem \ref{quniquedistance}, $q(G\cp H)\ge  (q(G)-1)+1=q(G)$.\epf

%We begin with a lower bound tool that is not restricted to graph products.

\begin{corollary}\label{PsboxP2}
For any path $P_s$ on $s\ge 2$ vertices, $q(P_s\cp P_2)=s$.
\end{corollary}
\bpf By Proposition \ref{cPathLowerBound}, we have $s\leq q(P_s\cp P_2)$. Observe $P_s\cp P_2$ has a Hamilton cycle of order $2s$, so by Corollary \ref{HamCycle} we know $q(P_s\cp P_2)\leq s$.  Thus,    $q(P_s\cp P_2)=s$.\epf 

The matrix $\wh C_s$ obtained from the adjacency matrix of $C_s$ by changing the sign
on a pair of symmetrically placed ones is called the {\em flipped cycle matrix}; note that $\wh C_s$ has every diagonal entry equal to zero. Set $k=\lc\frac s 2\rc$.   The distinct eigenvalues of $\wh C_s$ are $\lam_j = 2 \cos \frac{\pi (2j-1)}s, j=1,\dots,k$, each with multiplicity two  except that $\lam_k=-2$ has multiplicity one when $s$ is odd \cite{AIM08}. 

\begin{proposition}\label{prop:C4cartprod} Let $G$ be a graph of order $t$.  If there exists a  matrix $A\in\sym(G)$  such that $q(A)=q(G)$ and $-\dev(A)=\dev(A)$, then  $q(C_4\cp G)\le q(G)+1$.  If in addition $0\not\in\dev(A)$, then $q(C_4\cp G)\le q(G)$. \end{proposition}
\bpf  Assume $A\in\sym(G)$, $q(A)=q(G)$, and $-\dev(A)=\dev(A)$.
  Define \[M=\mtx{A & I_t & 0 & -I_t \\
 I_t & -A & I_t & 0 \\
 0 & I_t & A & I_t \\
 -I_t & 0 & I_t & -A}\!\!, \] so \[M^2=\mtx{A^2+2I_t & 0 & 0 & 0 \\
 0 & A^2+2I_t & 0 & 0 \\
 0 & 0 & A^2+2I_t & 0 \\
 0 & 0 & 0 & A^2+2I_t}\!\!.\]
This implies $\dev(M)\subseteq S:=\{\pm\sqrt{\lam^2+2}:\lam\in\dev(A)\}$. If $0\not\in\dev(A)$, then $|\{{\lam^2+2}:\lam\in\dev(A)\}|=\frac {q(A)}2$ and $|S|=q(A)$.   If $0\in\dev(A)$, then $|\{\sqrt{\lam^2+2}:\lam\in\dev(A)\}|=\frac {q(A)+1}2$ and $|S|=q(A)+1$.   Observe that $M=\wh C_4\otimes I_t+D\otimes A$ for $D=\diag(1,-1,1,-1)$.   %$B=\mtx{0 & 1 & 0 & -1 \\ 1 & 0 & 1 & 0 \\ 0 & 1 & 0 & 1 \\ -1 & 0 & 1 & 0}.$  
 Thus, $M\in\sym(C_4\cp G)$. \epf

The next result shows that the bound in Proposition \ref{prop:C4cartprod} is tight.
\begin{corollary}\label{P2sboxC4} For $k\ge 1$, $s\ge 2$ and $s\not\equiv 2 \mod 4$, 
\bit
\item $q(C_4\cp P_{2k})= 2k$, and
\item $q(C_4\cp C_{s})= \lc\frac s 2\rc$.
\eit \end{corollary}
\bpf We present upper and lower bounds that are equal to the stated value.  For the upper bound we apply Proposition \ref{prop:C4cartprod}: Use the adjacency matrix $A$ for $G=P_{2k}$, and note that $-\spec(A)=\spec(A)$ and $0\not\in\spec(A)$. Use the flipped cycle matrix $\wh C_s$ for $G=C_{s}$, and note that $-\spec(\wh C_s)=\spec(\wh C_s)$, and  $0\not\in\spec(A)$ if $s\not\equiv 2 \mod 4$.
 For $P_{2k}$, Proposition \ref{cPathLowerBound} provides the lower bound.  Since $\M(C_4\cp C_{s})\le 8$,\footnote{This is well known (and is immediate from \cite[Proposition 2.4 and Corollary 2.8]{AIM08}).}  Observation \ref{M} provides the lower bound for $C_{s}$.
\epf

%\begin{theorem}\label{thm:CPbipart} Suppose $G$ is bipartite, there exists a unitary matrix $A\in\sym(G)$ with zero diagonal, and there exists $A'\in\sym(G')$ such that $q(A)=q(G)$ and $-\spec(A)=\spec(A)$.  Then  $q(G\cp G')\le q(G')+1$.  If in addition $0\not\in\spec(A)$, then \end{theorem}

%Observe that Theorem \ref{thm:CPbipart} applies to $G=K_{m,n}$
%\begin{corollary}  Apply this thm \end{corollary}

%%%%%%%%%%%%%%%%%%%%%%%%%%%%%%%%%%%%%%%% Tensor Results
\subsection{Tensor products}\label{tnr}

The {\em tensor product} of graphs $G$ and $G'$, denoted $G\x G'$, has vertex set $V(G)\times V(G')$ and edge set $\{(u,u')(v,v') : uv\in E(G) \text{ and } u'v'\in E(G')\}$.

%\begin{remark}\label{kroneckerproduct} Let $A$ and $B$ be square  real matrices (possibly of different sizes). Then $\spec(A\otimes B)= \spec(A)\spec(B)$ \cite[Theorem 4.8]{Zhang}.  Thus $\dev(A\otimes B)= \dev(A)\dev(B)$.  \end{remark}
\begin{remark} \label{PstimesP2}
	For $s\ge 2$, the graph $P_s\times P_2$ is two (disjoint) copies of $P_s$, so $q(P_s\times P_2)=s$.   %  The upper bound also follows from Proposition \ref{thm:tpp}.
	\end{remark}

\begin{proposition}\label{prop:tp}  Let $G$ and $G'$ be  connected graphs. Let $A=[a_{ij}]\in\sym(G)$ with a zero diagonal and $A'=[a'_{ij}]\in\sym(G')$ with a zero diagonal. Then $A\otimes A'\in \sym(G\x G')$. \end{proposition}

\bpf Let $u\in$ $V(G)$ and $u'\in$ $V(G')$.  Then, the vertices of $G\x G'$ are $(u,u')$ and the edges are $(u,u')(v,v')$ where $uv$ and $u'v'$ are edges in $G$ and $G'$, respectively. Since $a_{uu}=a'_{u'u'}=0$,    $a_{uv}$ and $a'_{u'v'}$ are both nonzero if and only if $uv\in E(G)$ and $u'v'\in E(G')$.  Thus, $(A\otimes A')\in \sym(G\x G')$.  \epf %Thus, $(u,u')(v,v')\in E(G\x G')$ $\iff$ 
\vspace{-5pt}
\begin{proposition}\label{GXC4}  Let $G$ be a graph.  If there exists $A\in\sym(G)$ such that the diagonal of $A$ is zero, $q(A)=q(G)$, and $-\dev(A)=\dev(A)$, then $q(C_4\times G)\le q(G)$.  In particular:
\ben
\item $q(C_4\x P_s)= s$.
\item $q(C_4\x C_4)=2$.
\item $q(C_4\x C_{2k})\le k$.
\een
\end{proposition}
\bpf Assume the hypotheses.
  Define \[M=\frac 1{\sqrt 2}\mtx{ 0 & A & 0 & -A \\
 A & 0 & A & 0 \\
 0 & A & 0 & A \\
 -A & 0 & A & 0}\!\!, \mbox{ so }M^2=\mtx{A^2 & 0 & 0 & 0 \\
 0 & A^2 & 0 & 0 \\
 0 & 0 & A^2 & 0 \\
 0 & 0 & 0 & A^2}\!\!.\]
This implies $\dev{M}\subseteq \dev(A)\cup(-\dev(A))=\dev(A)$, so  $q(M)=q(A)$.   Let $B=\mtx{0 & 1 & 0 & -1 \\
 1 & 0 & 1 & 0 \\
 0 & 1 & 0 & 1 \\
 -1 & 0 & 1 & 0}\!\!.$  Then $M=B \otimes A\in\sym(C_4\x G)$ and $q(C_4\x G)\le q(G)$.
 
 Since $\spec(A)=-\spec(A)$ for $A$ the adjacency matrix of $P_s$ or $C_{2k}$, $q(C_4\x P_s)\le s$ and $q(C_4\x C_{2k})\le k$.  The specific results then follow from the general upper bound just established, and that $C_4\x P_s$ has a unique shortest path on $s$ vertices and $q(C_4)=2$. \epf

%\begin{remark} \label{rem:PstimesPt}
\vspace{-6pt}
The next result gives a  bound on the tensor product of two paths.	Since it is known that a path can be realized with any distinct spectrum, it would be reasonable to ask for a spectrum that behaves well under products, e.g., $\{1,2,4,\dots,2^{k-1}\}$ for $k=s,t$.  However, much less is known about what spectra can be realized by paths assuming a zero diagonal.  It is not true that a path can be realized with any spectrum and zero diagonal, because the sum of the eigenvalues must be zero.  %\end{remark}

\begin{proposition}\label{thm:tpp}  For the tensor product of paths, 
\[\min\{s,t\} \leq q(P_s\times P_t) \leq \left\{ \begin{array}{ll}
\frac{ts}{2}& \mbox{ for } s, t \text{ even}\\[1mm]
\frac{(t-1)s}{2}+1 & \mbox{ for }  s \text{ even}, t \text{ odd} \\[1mm]
\frac{(t-1)(s-1)}{2} +1& \mbox{ for }  s,t \text{ odd}
\end{array} \right.\]
\end{proposition}

\begin{proof} The lower bound is a direct application of Theorem \ref{quniquedistance}. 

For the upper bound, note that for paths the adjacency matrix achieves $q$. We can find the eigenvalues of $P_s\times P_t$ by multiplying all possible pairs of eigenvalues from the adjacency matrices for $P_s$ and $P_t$. As a path is bipartite, the adjacency eigenvalues of the path are symmetric about zero. We then count the eigenvalues.

% \[\begin{array}{c|ccccccc}
% & \cdots  & -\beta & -\alpha & 0 &\alpha & \beta  & \cdots\\
% \hline
% \vdots & \ddots\\
% -\gamma & & \beta\gamma & \alpha\gamma & 0 & -\alpha\gamma & -\beta\gamma \\ 
% -\delta && \beta\delta &\alpha\delta & 0 & -\alpha\delta & -\beta\delta\\
% 0 && 0 & 0 & 0 & 0 & 0\\
% \delta && -\beta\delta & -\alpha\delta& 0 &\alpha \delta & \beta\delta\\
% \gamma &&-\beta\gamma & -\alpha\gamma& 0 &\alpha \gamma & \beta\gamma\\
% \vdots &&&&&&& \ddots\\
% \end{array}.\] 

If $s$ and $t$ are both even, we have $\frac{t}{2}$ positive eigenvalues of $P_t$ and since the $s$ eigenvalues of $P_s$ are symmetric about zero, we have at most  $\frac{ts}{2}$ distinct eigenvalues for $P_s \times P_t$.

If s is even and t is odd, then there are $\frac{t-1}{2}$ distinct positive eigenvalues of $P_t$ and $s$ non-zero eigenvalues of $P_s$. Thus, we have at most  $\frac{(t-1)s}{2}$ distinct nonzero eigenvalues.  Since $t$ is odd, $P_t$ contains a zero eigenvalue, and so does $P_s \times P_t$. Therefore we add $1$ to our bound. 

If s and t are odd, then there are $\frac{t-1}{2}$ distinct positive eigenvalues of $P_t$ and $s-1$ non-zero eigenvalues of $P_s$. Thus we have at most  $\frac{(t-1)(s-1)}{2}$ distinct nonzero eigenvalues.  Since $t$ is odd, $P_t$ contains a zero eigenvalue, and so does $P_s \times P_t$. Therefore we add $1$ to our bound. 
 \end{proof}

%---------------------
\subsection{Strong products}\label{sp}

The {\em strong product}\,\footnote{The {\em strong} in strong product has no connection with the {\em strong} in Strong Multiplicity Property (or Strong Spectral Property).} of graphs $G$ and $G'$, denoted $G\boxtimes G'$, has vertex set $V(G\boxtimes G')=V(G)\times V(G')$ and edge set 
\bea E(G\boxtimes G')\!\!&=&\!\! \{(u,u')(v,v') : u=v \text{ and } u'v'\in E(G') \}
\\
& &\!\!\cup\ \{(u,u')(v,v') :   u'=v' \text{ and } uv\in E(G)\}
\\
& &\!\!  \cup\ \{(u,u')(v,v') : u'v'\in E(G') \text{ and } uv\in E(G)\}.
\eea
That is, $E(G\boxtimes G') = E(G\x G') \cup E(G\cp G')$.

\begin{proposition}\label{prop:sp}  Let $A\in\sym(G)$ and $A'\in\sym(G')$ with both having every diagonal entry nonzero.  Then $A\otimes A'\in \sym(G\boxtimes G')$. \end{proposition}

\bpf Let $D_A$ denote the matrix containing the diagonal of $A$ and similarly for $D_{A'}$. We observe that
\bea
A\otimes {A'} 	&=& (A-D_A + D_A) \otimes ({A'}-D_{A'} +D_{A'})\\
%			&=& (A-D_A)\otimes ({A'}-D_{A'} + D_{A'}) + D_A \otimes  ({A'}- D_{A'} +D_{A'})\\
            &=& (A-D_A)\otimes({A'}-D_{A'}) + (A-D_A)\otimes D_{A'} +\\
            & & D_A\otimes ({A'}-D_{A'}) +D_A \otimes D_{A'}.
\eea
Observe that  $(A-D_A)\otimes({A'}-D_{A'})$ gives the edges of  $G\x G'$ by Proposition \ref{prop:tp}. The edges  $G\cp G'$ are given by $(A-D_A)\otimes D_{A'}+D_A\otimes ({A'}-D_{A'})$. We note that as the Cartesian and tensor products of  graphs have no common edges,  so there is no cancellation, and that adding the preceding matrices gives us the off-diagonal nonzero pattern of $G\boxtimes G'$.  Adding $D_A \otimes D_{A'}$ will not affect the off-diagonal pattern. Therefore, $A\otimes {A'} \in \mathcal{S}(G\boxtimes G')$. 
\epf\vspace{-5pt}

%\begin{corollary} Suppose $P_s$ is a subgraph of  of $G$, $P_{s'}$ is a subgraph of  of $G'$, and the orders of $G$ and $G'$ are  $s$ and  $s'$, respectively.  Then \[ q(G\otimes G') \leq  \left\{ \begin{array}{ll}s+s'-2& \mbox{ for } s, s' \text{ even}\\s+s'-1 & \mbox{ otherwise}\end{array} \right.\]\end{corollary}

\begin{proposition}\label{prop:spp2}  Let $G$ be a graph.
If  $A\in\sym(G)$, every diagonal entry of $A$ is nonzero, $q(A)=q(G)$, and  $(-\dev(A)=\dev(A)$ or  $0\in\dev(A))$, then $q(G\boxtimes P_2)\leq q(G)$.
\end{proposition}\vspace{-3pt}

\bpf  Assume $A\in\sym(G)$, every diagonal entry of $A$ is nonzero, and $q(A)=q(G)$.
%We can realize the spectra   $\{-1,1\}$ and $\{0,1\}$ for $P_2$ with the following nonzero diagonal matrices: 
If $-\dev(A)=\dev(A)$, choose $B=\frac 1 {\sqrt{2}}\mtx{1 & 1\\ 1 & -1}$, so $\spec(B)=\{-1,1\}$.  If $0\in\dev(A)$, choose $B=\frac 1 2\mtx{1 & 1\\ 1 & 1}$, so $\spec(B)=\{0,1\}$.  Then,
 $A\otimes B\in \sym(G\boxtimes P_2)$ and  $\dev(A\otimes B)=\dev(A)\dev(B)=\dev(A)$. Therefore $q(G\boxtimes P_2)\leq q(A\otimes B)\leq q(G)$.
\epf\vspace{-5pt}

%\begin{corollary}$q(P_{2s}\boxtimes P_2)\leq 2s$.\end{corollary}
%\bpf By Proposition  \ref{nzdiag}, there is a matrix $A$ with nonzero diagonal and $\spec(A)=\{\pm 1, \dots, \pm s\}$.  So the bound given in Proposition \ref{prop:spp2} applies.  \epf

\begin{proposition}\label{spP3}
Let $A\in\sym(G)$ with every diagonal entry nonzero such that $q(A)=q(G)$ and $\dev(A)=-\dev(A)$.  Then
\[ q(G\boxtimes P_3) \leq \left\{ \begin{array}{ll}
q(G) +1 & \text{ if $0 \notin \dev(A)$}\\
q(G) & \text{ if $0 \in \dev(A)$}\\
\end{array} \right.\!\!.\]
\end{proposition}\vspace{-3pt}

\bpf
We may realize the spectrum $\{-1,0,1\}$ for $P_3$ with the matrix \[B= \left[\begin{array}{rrr}
-\frac{5}{6}\sqrt{\frac{3}{5}} \, & -\frac{5}{6}\sqrt{\frac{3}{5}} \, & 0 \, \\
-\frac{5}{6}\sqrt{\frac{3}{5}} \, & \frac{1}{2}\sqrt{\frac{3}{5}} \, & \frac{2}{3}\sqrt{\frac{3}{5}} \, \\
0 \, & \frac{2}{3}\sqrt{\frac{3}{5}} \, & \frac{1}{3}\sqrt{\frac{3}{5}} \, 
\end{array}\right]\!\!,\]
which has every diagonal entry nonzero.  By similar reasoning as in Proposition \ref{prop:spp2}, $\dev(A \otimes B)= \dev(A) \cup 0$. The upper bound follows immediately. 
\epf \vspace{-5pt}

\begin{corollary}\label{P3boxP3}
$q(P_3 \boxtimes P_3)=3$.
\end{corollary}
\bpf We observe that 
$3\leq q(P_3 \boxtimes P_3)$ since  the diagonal vertices in $P_3\boxtimes P_3$ have a unique shortest path of length 2. Furthermore, $q(P_3 \boxtimes P_3)\leq 3$ by Proposition \ref{spP3}.
\epf

The next result is worse for odd paths than Proposition \ref{spP3} because Theorem \ref{nzdiag} does not apply when a zero eigenvalue is desired.

\begin{proposition}\label{prop:spp} For $s'\ge s\ge 2$, 
\[s \leq q(P_s\boxtimes P_{s'}) \leq %q_M(P_s\otimes P_{s'}) \leq 
\left\{ \begin{array}{ll}
s+s'-2& \mbox{ for } s, s' \text{ even}\\
s+s'-1 & \mbox{ otherwise}
\end{array} \right.\!\!.\]
%In particular, $q(P_{2k}\boxtimes P_2)=2k$.
 \end{proposition}

\bpf With the vertices of $P_s$ and $P_{s'}$ labeled by $\{1,\dots,s\}$ and $\{1,\dots,s'\}$, there is a unique shortest path in $P_s\boxtimes P_{s'}$ between vertices $(1,1)$ and $ (s,s)$, so $s \leq q(P_s\otimes P_{s'})$. By Theorem \ref{nzdiag}, for any  $\lam_1,\dots,\lam_n$, there is a matrix $B\in\sym(P_n)$ %with SSP 
and $\spec(B)=\{\lam_1,\dots,\lam_n\}$.     Choose $A\in\sym(P_s)$ with $\spec(A)=\{1,2,\dots, 2^{s-1}\}$ and $A'\in\sym(P_{s'})$ with $\spec(A)=\{1,2,\dots ,2^{s'-1}\}$.  Then $\dev(A\otimes A')=\{1,2,\dots,2^{s+s'-2}\}$, so $q(A\otimes A')=s+s'-1$.  In the case $s$ and $s'$ are both even, choose $\spec(A)=\{\pm 1,2,\dots, \pm 2^{s/2-1}\}$ and $A'\in\sym(P_{s'})$ with $\spec(A)=\{\pm 1,\pm 2,\dots, \pm 2^{s'/2-1}\}$.  Then $\dev(A\otimes A')=\{\pm 1,\pm 2,\dots,\pm 2^{s/2+s'/2-2}\}$, so $q(A\otimes A')\le s+s'-2$.
\epf

\vspace{-10pt}
%%%%%%%%%%%%%%%%%%%%%%%%%%%%%%%%%%%%%%%% 

\section{Other graph operations}\label{sotherops}

In this section we present results for block-clique graphs and for joins. 

\subsection{Block Clique-Graphs}\label{glue}

  Let $G=(V,E)$ and $G'=(V',E')$ be graphs.  The {\em union} of $G$ and $G'$  is the graph $G\cup G'=(V\cup V', E\cup E')$. If $ V\cap V'=\emptyset$, then the union is disjoint and can be denoted by $G\du G'$.  If $V\cap V'\ne \emptyset$, then the {\em intersection} of $G$ and $G'$  is the graph $G\cap G'=(V\cap V', E\cap E')$.  If $ V\cap V'=\{v\}$, then $G\cup G'$ is called the {\em vertex sum} of $G$ and $G'$ and can be denoted by $G\oplus_v G'$; in this case $v$ is called the {\em summing vertex}.
A block-clique graph is constructed from cliques by a sequence of vertex sums.  In this section we establish the value of $q$ for two families of block-clique graphs, clique-paths and clique-stars,  which we define below.  

\begin{definition} For $s\ge 2$ and $n_{s_i}\ge 2$ for $ i=1,\dots,s$, we define a graph $KP(n_1,n_2,\ldots,n_s)$, called a \textit{clique-path}, to be a graph constructed by vertex sums using distinct summing vertices and cliques $K_{n_1}, K_{n_2},\dots,K_{n_s}$ in order.
\end{definition}

\begin{definition} For $s\ge 2$ and $n_{s_i}\ge 2$ for $ i=1,\dots,s$, we define a graph $KS(n_1,n_2,\ldots,n_s)$, called a \textit{clique-star}, to be a graph constructed by vertex sums using only one summing vertex and cliques $K_{n_1}, K_{n_2},\dots,K_{n_s}$.
The vertex that is in every clique is called the {\em center} and every other vertex is called {\em noncentral}.  
\end{definition}

Of course, $KP(n_1,n_2)=KS(n_1,n_2)$.

\begin{theorem}\label{cliquepath} For $s\ge 2$ and $n_{i}\ge 2, i=1,\dots,s$,  $q(KP(n_1,n_2,\ldots,n_s)) = s+1$.
\end{theorem}

\bpf  We observe that there is a unique shortest path between the first summing vertex and the last summing vertex. We can extend this path by 2 vertices, one in $K_{n_1}$ and one in $K_{n_s}$ to find a unique path of length $s$. Thus, $q(KP(n_1,n_2,\ldots,n_s)) \geq s+1$ by Theorem \ref{quniquedistance}.

For the reverse inequality,  number the vertices of $K_{n_i}$ consecutively in order of the cliques, with the first summing vertex in $K_{n_i}$ as first and the second summing vertex in $K_{n_i}$ last among the vertices of $K_{n_i}$ for $i=2,\dots,n-1$; the summing vertex of $K_{n_1}$ is last and the summing vertex of $K_{n_s}$ is first among vertices in these cliques. Then the matrix 
\[\small A \!=\! \mtx{
J_{n_1-1} & \bo_{n_1-1} & &&&&\\
\bo_{n_1-1}^T	&2 &\bo_{n_2-2}^T & 1& \\
  & \bo_{n_2-2} &J_{n_2-2}& \bo_{n_2-2}\\
&  1&\bo_{n_2-2}^T	&2 &\ddots \\
& & & \ddots & \ddots&\\
& & & & & 2 & \bo_{n_s-1}^T\\
& & & & & \bo_{n_s-1} & J_{n_s-1}
}\!\!\in\!\sym(KP(n_1,,\ldots,n_s)).\]
Since $\rank A=s$,  $q(KP(n_1,n_2,\ldots,n_s))\leq s+1$  by Theorem \ref{mr+1}.
\epf\vspace{-5pt}

% \begin{theorem} We observe that $q(\mathcal{S}(3,3,3))=3$.\end{theorem}

% \begin{proof} There is a unique path of length 2 through the from any vertex in one of the $K_3$'s to any vertex in another $K_3$ through the identified vertex, which gives a lower bound of $2+1=3$ by Theorem \ref{quniquedistance}. We observe that the normalized laplacian has 3 distinct eigenvalues, which gives the upper bound.
% \end{proof}

\begin{theorem}\label{cliquestar} For all $s\ge 2$ and $n_{s_i}\ge 2, i=1,\dots,s$,  the clique-star $G:=KS(n_1,n_2,\ldots,n_s)$ has $q(G)=3$.\end{theorem}
\bpf  Let  $\ell_i=n_i-1$ (the cardinality  of the set of noncentral vertices of $K_{n_i}$), $n=1+\sum_{i=1}^s \ell_i$ (the order of $G$), and number the noncentral vertices of $K_{n_i}$ consecutively in order of the cliques, with the center last (vertex $n$).
There is a unique path of length two  from any noncentral vertex in one  $K_{n_i}$ to any noncentral vertex in another $K_{n_j}$ ($j\ne i$) through the center 
vertex  $n$, so $q(G)\geq 3$ by Theorem \ref{quniquedistance}.

Define   $\nJ k=\frac 1 k J_k$, $\nbo k=\frac 1 k \bo_k$, and 
\[A=\mtx{
\nJ{\ell_1}& &&\nbo{\ell_1}\\
&\ddots&& \vdots\\
&&\nJ{\ell_s}& \nbo{\ell_s}\\[2pt]
\nbo{\ell_1}^T & \cdots & \nbo{\ell_s}^T & \sum_{i=1}^s\frac 1 {\ell_i}
}\in\sym(G).\]
We show that $q(A)=3$, implying $q(G)=3$.

Observe that $A(n)=\nJ{\ell_1}\oplus\cdots\oplus\nJ{\ell_s}$.  We can construct  $A$  from  $A(n)$ by taking the sum of one row associated with each $K_{n_i}$ to form a new last row, and then adding the corresponding columns of this $n\x (n-1)$ matrix to form a new last column.  Thus  $\rank A=\rank A(n)=s$, which implies   $\mult_A(0)=n-s$.   Since $\bx_i:=[\bzero^T,\dots,\bzero^T,\nbo{\ell_i}^T,\bzero^T,\dots,\bzero^T]^T$ is an eigenvector for eigenvalue 1 of $A(n)$,  $\mult_{A(n)}(1)\ge s$.  By  interlacing  (Theorem \ref{thm:interlacing}), $\mult_{A}(1)\ge s-1$, so $\mult_A(0)+\mult_A(1)\ge n-1$.  Since $q(A)\ge 3$, there is exactly one more eigenvalue (necessarily different from 0 and 1 and of multiplicity one) and $q(A)=3$. 
\epf

%-----------
\subsection{Joins}\label{jnr}

The {\em join} of disjoint graphs $G=(V,E)$ and $G'=(V',E')$, which is denoted by $G\vee G'$, has vertex set $V\cup V'$ and edge set $E\cup E' \cup \{vv' : v\in V, v'\in V'\}$.
It was shown in \cite{AACFMN13} that $q(\overline{K_n}\vee \overline{K_n})=2$ and $q(\overline{K_n}\vee \overline{K_m})=3$ for $n\ne m$ (see Theorem \ref{qpaperKmn}).  Monfared and Shader  showed in \cite{MS16} that $q(G\vee H)=2$ for connected graphs $G$ and $H$   of the same order (see  Theorem \ref{qjoinsamesize}). 
The next example shows that a join can require an arbitrarily large number of distinct eigenvalues.

\begin{example}\label{PsveeK1}  Since $P_s\vee K_1$ has an induced $P_s$, $\mr(P_s\vee K_1)\ge \mr(P_s)=s-1$, and since $P_s\vee K_1$ is not a path, $\mr(P_s\vee K_1)\le s-1$.  Thus $\M(P_s\vee K_1)=2$, which implies $q(P_s\vee K_1)\ge \lc\frac{s+1}2\rc$ by Observation \ref{M}.  Since $P_s\vee K_1$ has a Hamilton cycle, $q(P_s\vee K_1)\le \lc\frac{s+1}2\rc$ by Corollary \ref{HamCycle}.
\end{example}

\begin{theorem} Let $G$ and $G'$ be connected graphs such that $|V(G)|=n$ and  $|V(G')|=n-\ell$ for some $1\leq \ell< n$. Then $q(G\vee G')\leq 2+\ell$. \end{theorem}

\bpf
Create a graph $G''$ by adding new vertices $v_1,\dots,v_\ell$ to $G'$ and adding some combination of possible edges involving these vertices to make $G''$ connected. Then $G\vee G'$ is a subgraph of $G\vee G''$.   By Theorem \ref{qjoinsamesize} we have $q(G\vee G'')=2$ and
there is a matrix $A\in\sym(G\vee G'')$ with two eigenvalues each of multiplicity $n$. Then  $\lambda_1(A)=\dots=\lambda_n(A)<\lambda_{n+1}(A)=\dots=\lambda_{2n}(A)$. By deleting rows and columns of $A$ corresponding to the new vertices $v_1,\dots,v_\ell$, we obtain a principal submatrix $B\in\sym(G\vee G')$. %Let $\spec(B)=\mu_1\leq\dots\leq\mu_{2n-k}$. 
Then by eigenvalue interlacing (Theorem \ref{thm:interlacing}), we have 
\[\lambda_1(A)\leq \lam_1(B)\leq \dots\leq\lam_{n-\ell}(B)\leq\lambda_n(A)=\lambda_1(A)\]
\[\lambda_{1}(A)=\lambda_{n-\ell+1}(A)\leq\lam_{n-\ell+1}(B)\leq\dots\leq \lam_{n}(B) \leq \lambda_{n+\ell}(A)=\lam_{2n}(A)
\]
\[\lambda_{n+1}(A)\leq \lam_{n+1}(B)\leq\dots \leq\lam_{2n-\ell}(B)\leq \lambda_{2n}(A)=\lambda_{n+1}(A).\]
This gives us $\lambda_1(A)=\lam_1(B)= \dots=\lam_{n-\ell}(B)$ and $\lambda_{n+1}(A)=\lam_{n+1}(B)= \dots=\lam_{2n-\ell}(B)$. The remaining $\ell$ eigenvalues are
bounded such that $\lambda_{1}(A)\leq\lam_{n-\ell+1}(B)\leq\dots\leq \lam_{n}(B) \leq \lambda_{2n}(A)$. Therefore $q(G\vee G')\leq 2+\ell$.
\epf
\vspace{-3pt}

 %%%%%%%%%%%%%%%%%%%%%%%%%%%%%%%%%%%%%%% 
\section{Summary of the impact of graph operations}\label{graphops} 

In this section we provide some new examples illustrating
 the impact of graph operations on $q$ and  summarize what is known about the impact of other operations. If we say that an operation $\cdot$ on  two graphs $G$ and $H$ raises $q$, this means that $q(G\cdot H)> \max\{q(G), q(H)\}$.  Saying that  $\cdot$ on   $G$ and $H$ lowers $q$ means that $q(G\cdot H)< \min\{q(G), q(H)\}$, whereas saying $\cdot$ maintains $q$ means that $q(G\cdot H)=q(G)= q(H)$.  The meaning of {\em raises, lowers}, and {\em maintains} is clear when the operation is on a single graph.

%\begin{remark}\label{joinlow} 
It is clear from Theorem \ref{qjoinsamesize} that the join operation is capable of decreasing $q$; for example, $q(P_n\vee P_n)=2$ but $q(P_n)=n$.  Of course, the join can also maintain $q$. %;  for example, $q(P_2\vee P_2)=2$ but $q(P_2)=2$. 
To see that the join can raise $q$, define the {\em $d$th hypercube}  recursively by $Q_1=P_2$ and $Q_d=Q_{d-1}\cp P_2$.  The vertices of $Q_d$ are written as strings of  zeros and ones of length  $d$, and two vertices are adjacent if and only if  they differ in exactly  one place.
%\end{remark}

\begin{proposition}\label{raiseq} The join can raise the value of $q$, because $q(Q_5\vee P_2)\geq 3$. \end{proposition}

\bpf
The set of vertices $S=\{00000,00111,11110\}$ in $Q_5$ is an independent set of $Q_5\vee P_2$. The only common neighbors of these vertices are $v_1,v_2\in V(P_2)$. That is, 
 \[\left|\bigcup_{u,w\in S, u\ne w} (N(u)\cap N(w))\right| = |\{v_1,v_2\}| =2 < 3=|S|.\] By the contrapositive of Theorem \ref{independentsets}, we have $q(Q_5\vee P_2)\neq 2$ and therefore $q(Q_5\vee P_2)\geq 3$.  Note that $q(Q_5)=2$ \cite[Corollary 6.9]{AACFMN13} and $q(P_2)=2$.
\epf

Let $G=(V,E)$ be a graph.  
For $e\in E$, the notation $G-e$ means the result of deleting edge $e$ from $G$. For $v\in V$, the notation $G-v$ means the result of deleting $v$ and all edges incident with $v$.  The {\em contraction} of edge $e=uv$ of $G,$ denoted by $G/e$,
is obtained from $G$ by identifying the vertices $u$ and $v$, deleting a loop if one arises in this process, and replacing any multiple edges by a single edge.  The subdivision of edge $e={uv}$ of $G$, denoted by $G_e$, is the graph obtained from $G$ by deleting $e$ and inserting a new vertex $w$ adjacent exactly to $u$ and $v$.

Examples are given in  \cite{AACFMN13} showing  that the difference between $q(G)$ and $q(G-v)$ and the difference between $q(G)$ and $q(G- e)$ can grow arbitrarily large in either direction as a function of the number of vertices. The construction of a main example can be done with vertex sums. Let $x$ and $y$ be two nonadjacent vertices of $C_4$, and denote the other two vertices by $w$ and $z$.  Suppose also that $x$ is an endpoint of one $P_{k+1}$ and $y$ is an endpoint of another $P_{k+1}$.  The graph $P_{k+1}\oplus_x C_4\oplus_y P_{k+1}$ is denoted by $S_{k,k}$ in \cite{AACFMN13}  and it is shown there that $q(S_{k,k})=k+2$ \cite[Lemma 6.6]{AACFMN13}.

%If the goal is simply to raise or lower $q$, then there are simpler examples. %Let $S(n_1,\dots,n_k)$ dentote ...  %The examples rely on the graph $S_{k,k}$, which can be constructed from a path on $2k+3$ vertices adding a new vertex $z$ adjacent to exactly the two neighbors of the middle vertex of the path.   It is shown in \cite[Lemma 6.6]{AACFMN13} that $q(S_{k,k})=k+2$.

\begin{remark}\label{vdel}
Deleting  the midpoint of  $P_{2k+1}$ creates $2P_{k}$ and lowers $q$. 
 Deleting a vertex from $K_n$ creates $K_{n-1}$ and maintains $q$. 
Deleting the vertex $z$ from $G=P_{k+1}\oplus_x C_4\oplus_y P_{k+1}$ results in a path on $2k+3$ vertices.  Since $q(G)=k+2$ and $q(P_{2n+3})=2k+3$, the deletion of $z$ has raised $q$.  
% Adding a vertex adjacent to an endpoint of   $P_n$ creates $P_{n+1}$ and raises $q$. 
 %Adding a vertex adjacent to a vertex of degree 2 of   $P_n$ maintains $q$ (see Theorem \ref{thm:qisGminus1}). 
%Adding a vertex adjacent to both endpoints of   $P_n$ creates $C_{n+1}$ and lowers $q$ (for $n\ge 3$). 
  \end{remark}

\begin{remark}\label{edel} 
Deleting the middle edge from  $P_{2k}$ creates $2P_k$ and lowers $q$. 
  Deleting an edge from $K_n$ creates $K_{n}-e$ and maintains $q$ for $n\ge 4$ (see Proposition \ref{Kn-e}). 
Deleting the edge $xz$ from $G=P_{k+1}\oplus_x C_4\oplus_y P_{k+1}$ results in $S(k+2,k,1)$, a path with an extra leaf.  Since $q(G)=k+2$ and $q(S(k+2,k,1))=2k+3$, the deletion of $xz$ has raised $q$.  \end{remark}

\begin{remark}\label{econtract} Contracting an edge of  $P_n$ creates $P_{n-1}$ and lowers $q$.  Contracting  an edge of $K_n$ creates $K_{n-1}$ and maintains $q$ (for $n\ge 3$).   Contracting  the edge $e=xz$ of $G=P_{k+1}\oplus_x C_4\oplus_y P_{k+1}$ results in a generalized bull $GB(k,k)$.  Thus  $q(G)=k+2$ and $q(G/e)=2k+2$, raising $q$.
\end{remark}

\begin{remark}\label{esubdiv} Subdividing the edge $\{k+1,k+3\}$ of $GB(k,n-k-3)$ creates $P_{k+1}\oplus_x C_4\oplus_y P_{k+1}$ and lowers $q$.  Subdividing an edge of $C_{2k+1}$ maintains $q$ because $q(C_{2k+1})=k+1=q(C_{2k+2})$.  Subdividing a cycle edge  of $P_{k+1}\oplus_x C_4\oplus_y P_{k+1}$ creates a unique shortest path on $2k+3$ vertices and raises $q$.  
\end{remark}

Table \ref{tab:q-graphops} summarizes the possible effect on $q$ of various graph operations. 

%
%\vspace{-20pt}
\begin{table}[h!]
\caption{Possible effects on $q$ of various graph operations.  A column headed \# gives the result \# that describes the example illustrating lowering, maintaining, or raising $q$.\label{tab:q-graphops}}\vspace{-5pt}
{\small \begin{tabular}{|c|c|c|c|c|c|c|}\hline
Operation & Lower $q$  & \#   & Maintain $q$ &  \# &Raise $q$ &  \# \\
\hline
Join				& $P_n\vee P_n, n\ge 3$&	\ref{qjoinsamesize}	& $P_2\vee P_2$	&	\ref{qjoinsamesize}	& $Q_5 \vee P_2$ &\ref{raiseq}	\\\hline
Cartesian \!Product 	& 	&				& $P_s\cp P_2 $ &	\ref{PsboxP2}	& &	\\\hline
Tensor 	Product 	& 		&			& $C_4\times P_s$&	\ref{GXC4}	& $K_3\! \times\! P_2\!=\!C_6$&\ref{qpaperKn}	\\\hline
Strong Product		& 		&		& $P_3\boxtimes P_3$& \ref{P3boxP3}&  &\\
\hline
Vertex Sum	&	& &		$KS(n,n,n)$ & \ref{cliquestar} &	$KP(3,3)$ &\ref{cliquepath} \\\hline
Vertex Deletion	& $P_n$ &\ref{vdel}	&	$K_n$ &	\ref{vdel} & $C_n$ & \ref{qpaperKn}	\\\hline
Edge Deletion	& $P_n$ & \ref{edel}	&	$K_n$ &\ref{edel} & $C_n$ & \ref{qpaperKn}	\\\hline
Edge Contract	& $P_n$ &\ref{econtract}	&	$K_n$ &	\ref{econtract} & $P_{k+1}\oplus_x C_4\oplus_y P_{k+1}$ & \ref{econtract}	\\\hline
Edge Subdivide 	& $GB(k,n-k-3)$ & \ref{esubdiv} & 	$C_{2k+1}$	 &\ref{esubdiv} & $P_{k+1}\oplus_x C_4\oplus_y P_{k+1}$ & \ref{esubdiv}	\\\hline
\end{tabular}}
\begin{center}
\end{center}\end{table}
%

%%%%%%%%%%%%%%%%%%%%%%%%%%%%%%%%%%%%%%%% 
\section{Values of $q$ for graphs of order at most $6$}\label{s6}

The IEPG has been solved for all connected graphs of order at most $4$ in \cite{BNSY14} and order 5 in  \cite{IEPG2}.  Solution of the IEPG establishes the value of $q$; the results  for all connected graphs of order at most 5 are summarized in Table \ref{tab:ord5q}.  In this section we apply our previous results and additional ideas  to determine $q$ for all connected graphs of order 6 (see Table \ref{tab:ord6q}).\footnote{Ahn, Alar, Bjorkman, Butler, Carlson, Goodnight, Harris, Knox, Monroe, and Wigal  have recently determined all possible ordered multiplicity lists for graphs of order 6; most of their work is independent but in a few cases they cite results from this paper.}  

Note that if a graph $G$ is disconnected with   connected components $G_i, i=1,\dots,c$ then $q(G)=\max\{q(G_i): i=1,\dots, c\}$  and the solution to the  IEPG for $G$ can be deduced immediately from the solutions for each $G_i$, so data is customarily provided only for connected graphs.  All graphs are numbered using the notation in  {\em Atlas of Graphs} \cite{atlas}.

%\vspace{-5pt}
\begin{table}%[h!]
\begin{center}
\caption{Values of $q$ for graphs of order at most $5$ (using the graph numbering in \cite{atlas}).  All values of $q$ can be determined from the information in \cite[Figure 1]{IEPG2}.  \label{tab:ord5q}}\vspace{-3pt}
{\small\begin{tabular}{|c|c||c|c||c|c||c|c||c|c|}
\hline
$G\#$ & $q(G\#)$   & $G\#$ & $q(G\#)$ & $G\#$ & $q(G\#)$ & $G\#$ & $q(G\#)$ & $G\#$ & $q(G\#)$ 
 \\
\hline\hline
$G1$ &1 & $G3$ &	2 & $G6$ & 3 & $G7$ & 2 & $G13$ & 3
\\
\hline					
$G14$ &4 & $G15$ & 3 & $G16$ & 2 & $G17$ & 2 & $G18$ & 2
\\
\hline					
$G29$ & 3 & $G30$ & 4 & $G31$ & 5 & $G34$ & 3 & $G35$ & 4
\\
\hline					
$G36$ & 4 & $G37$ & 3 & $G38$ & 3 & $G40$ & 3 & $G41$ & 3
\\
\hline					
$G42$ & 3 & $G43$ & 3 & $G44$ & 3 & $G45$ & 3 & $G46$ & 3
\\
\hline					
$G47$ & 3 & $G48$ & 2 & $G49$ & 2 & $G50$ & 2 & $G51$ & 2
\\
\hline$G52$ & 2 &  &  &  &  &  &  &  & 
\\
\hline\end{tabular}}
\end{center}
\end{table}

We begin by establishing ordered multiplicity lists attaining the minimum value of $q$ that are attainable with SSP or SMP for some specific  graphs.  We then apply those results to determine $q$ for other graphs by using Observation \ref{thm:sub}. 
In many cases there is more than one way to establish the result, and in a few cases (most notably $K_6=G208$) the result is already known.  However,  we have grouped graphs by a subgraph having a matrix with SMP (or SSP, which implies SMP) for efficiency.  
We begin with graphs having $q(G)=3$.  Oblak and \v Smigoc \cite[Example 4.8]{OS14} provide the matrix $M_{96}$ in the next lemma and state its spectrum $\{-1,-1,0,0,2,2\}$.   

\begin{lemma}\label{222forG96} 
The matrix \[M_{96}= \mtx{0 & 0 & \frac{1}{\sqrt{2}} & 0 & 0 & 0 \\
 0 & 0 & \frac{1}{\sqrt{2}} & 0 & 0 & 0 \\
 \frac{1}{\sqrt{2}} & \frac{1}{\sqrt{2}} & 1 & -\frac{1}{\sqrt{2}} & 0 &
   -\frac{1}{\sqrt{2}} \\
 0 & 0 & -\frac{1}{\sqrt{2}} & 0 & 1 & 0 \\
 0 & 0 & 0 & 1 & 1 & -1 \\
 0 & 0 & -\frac{1}{\sqrt{2}} & 0 & -1 & 0}\] has SSP, $\oml(M_{96})=(2,2,2)$, and  $M_{96}\in\sym(G96)$.  Furthermore, $q(G96)=q_M(G96)=3$.
\end{lemma}
\bpf It can be verified by computation that $M_{96}$  has SSP  (see  \cite{suppdocs}, where Theorem \ref{thm:SSP-nec-suff-conds} is applied to $M_{96}$).  Since there is a unique shortest path on three vertices, $q(G96)=q_M(G96)=3$.
\epf\vspace{-5pt} 

\begin{corollary}\label{222fromG96}  %The matrix 
The following graphs $G$ have $q(G)=q_M(G)=3$: $G111$, $G114$, $G118$, $G121$, $G126$, $G133$, $G135$, $G136$, $G137$, $G140$, $G141$, $G144$, $G145$, $G149$, $G150$, $G156 -G159$,  $G161 - 167$, $G169 - 173$, $G177 - 180$, $G182 - 185$, $G193$. 
\end{corollary}
\bpf Each graph $G$ has $G96$ as a spanning subgraph, so by Lemma \ref{222forG96} and Observation \ref{thm:sub}, $q(G)\le q_M(G96)=3$.  With three exceptions,  each $G$ has a unique shortest path on three vertices,  and so has $q(G)=3$ by Theorem \ref{quniquedistance}.  

The exceptions are $G161$, $G170$, and $G179$.  In each of these cases we exhibit a set of independent vertices without enough common neighbors, so $q(G)\ne 2$ by Theorem \ref{independentsets}.  The vertices are numbered as in Figure \ref{fig:nbrhd}.\\
$G161$: The set  $\{3,4,5,6\}$ is an independent set of four vertices, but the union of neighborhood intersections is $\{1,2\}$.\\
$G170$: The set  $\{3,4,6\}$ is an independent set of three vertices, but the union of neighborhood intersections is $\{1,2\}$.\\
$G179$: The set   $\{3,4,5\}$ is an independent set of three vertices, but the union of neighborhood intersections is $\{1,2\}$.
\epf

\begin{figure}[!ht] 
\begin{center}\scalebox{.5}{\includegraphics{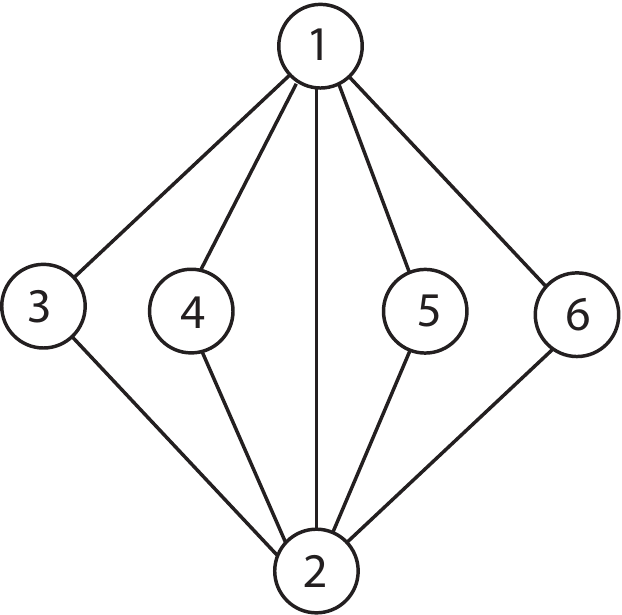}}\qquad\scalebox{.5}{\includegraphics{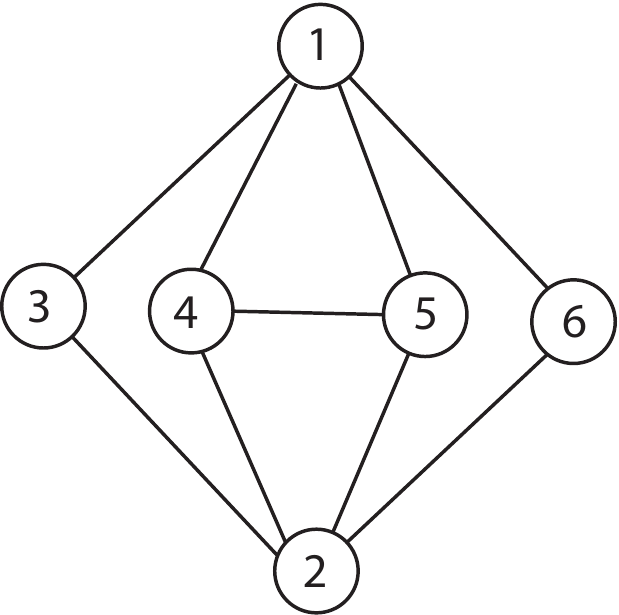}}\qquad\scalebox{.5}{\includegraphics{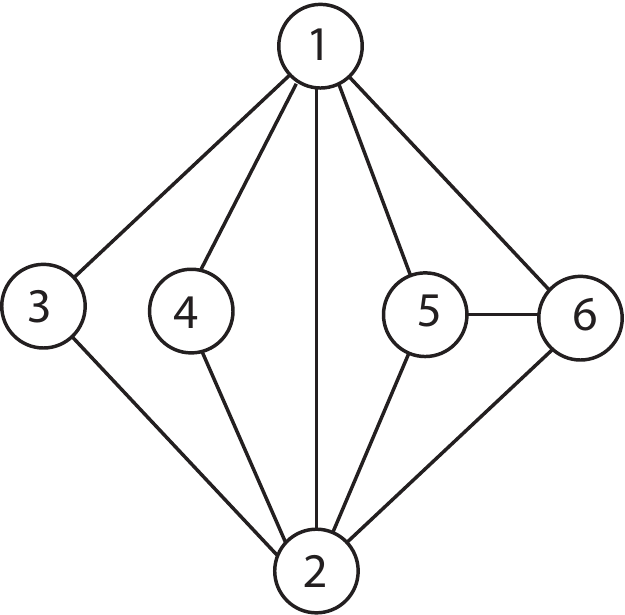}}\\
$G161$\qquad\qquad\qquad$G170$\qquad\qquad\qquad$G179$\\ \caption{Graphs to which the Theorem \ref{independentsets} is applied in the proof of Corollary \ref{222fromG96}}\label{fig:nbrhd}
\end{center}\vspace{-10pt}
\end{figure}

\begin{remark}\label{222forG105} %Since  $G105=C_6$, $q(G105)=3$ \cite[Lemma 2.7]{AACFMN13}.  
Since each of the  graphs $G=G105, G127, G147, G148, G151-G153$  has  a Hamilton cycle and each has a unique shortest path on three vertices, $q(G)=q_M(G)=3$. 
\end{remark}

\begin{lemma}\label{222forG125} The graph $G125$ has a matrix with SSP and ordered multiplicity list $(2,2,2)$.
\end{lemma}
\bpf  Observe that graph $G125$ can be constructed by adding a new vertex 6 adjacent to vertices 2 and 3 of the Banner = $G37$ (see Figure \ref{fig:graphsaug}).  It can be verified by computation  (see \cite{suppdocs}) that Goodnight's matrix \cite{Gpersonal}
\[M=\mtx{\frac{4}{3} & -\sqrt{\frac{2}{3}} & \sqrt{\frac{2}{3}} & 0 & 0 \\
 -\sqrt{\frac{2}{3}} & 0 & 0 & \frac{2}{3} & 0 \\
 \sqrt{\frac{2}{3}} & 0 & 0 & \frac{2}{3} & 0 \\
 0 & \frac{2}{3} & \frac{2}{3} & \frac{4}{3} & \frac{2}{3} \\
 0 & 0 & 0 & \frac{2}{3} & 0}\in\sym(G37)\] has  SSP and eigenvalues $\mu_1=-2/3$, $\mu_2=0$, and $\mu_3=2$ with multiplicities 2, 1, 2, respectively, so the ordered multiplicity list of $M$ is $(2,1,2)$.  Furthermore, the vector $\bx=[0,-\frac{1}{2},-\frac{1}{2},0,1]^T$ is a basis for the eigenspace of eigenvalue $\mu_2=0$.  Since $\supp(\bx)=\{2,3,5\}$,    $|\supp(\bx)\cap \{2,3\}|= 2$.  Therefore, we can apply the Augmentation Lemma (Lemma \ref{augment}) to obtain a matrix having eigenvalue $\mu_2=0$  with multiplicity $2$ and also eigenvalues $\mu_1$ and $\mu_3$ each with multiplicity 2.  Thus the graph $G125$ has a matrix with SSP and ordered multiplicity list $(2,2,2)$.
\epf

\begin{figure}[!ht] 
\begin{center}\scalebox{.5}{\includegraphics{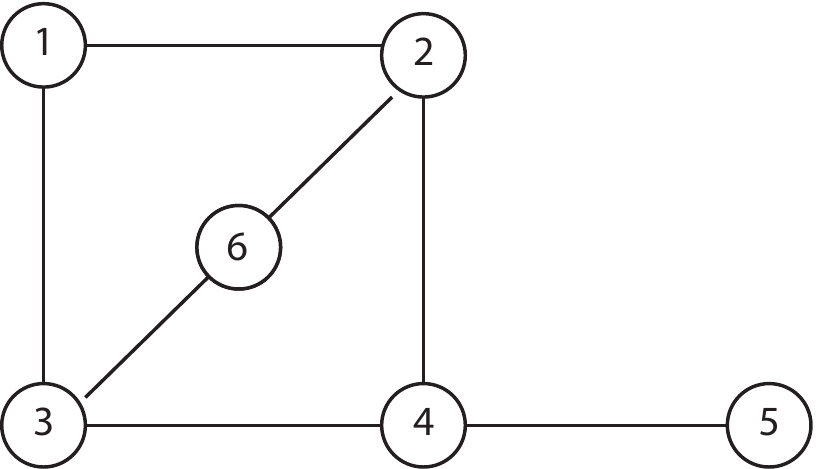}}\qquad\scalebox{.5}{\includegraphics{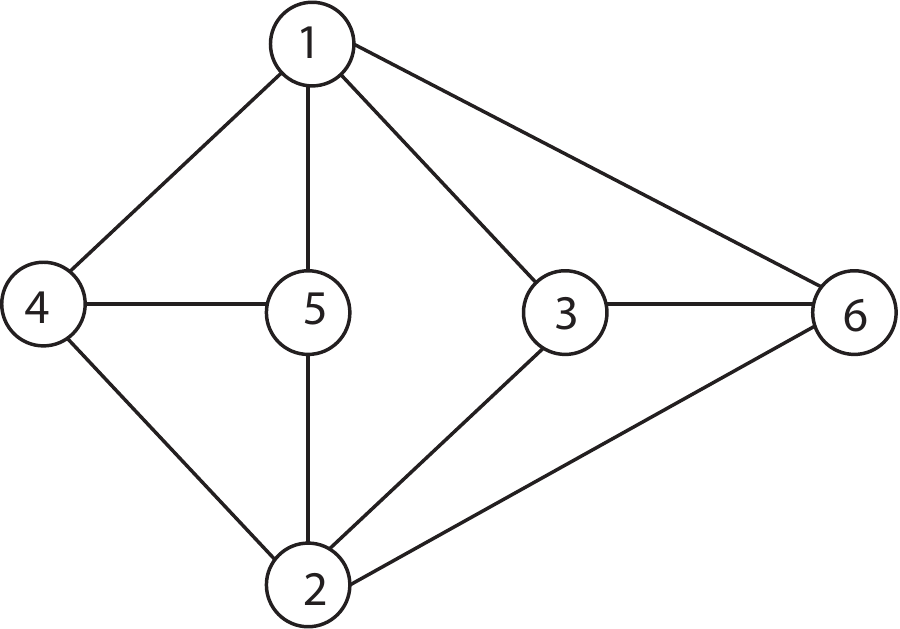}}\\ %\qquad\scalebox{.4}{\includegraphics{G194}}\\
$G125$\qquad\qquad\qquad$G190$\\ %\qquad\qquad\qquad$G194$\\ 
\caption{Graphs to which the Augmentation Lemma is applied}\label{fig:graphsaug}
\end{center}
\end{figure}

\begin{corollary}\label{222fromG125}  %The matrix 
The following graphs $G$ have $q(G)=q_M(G)=3$: $G138$, $G143, G160$. \end{corollary}
\bpf Each graph $G$ has $G125$ as a spanning subgraph, so by Lemma \ref{222forG125} and Observation \ref{thm:sub}, $q(G)\le q_M(G96)=3$.  Since each has a unique shortest path on three vertices,  each has $q(G)=q_M(G)=3$ by Theorem \ref{quniquedistance}.
\epf

\vspace{-5pt}
\begin{lemma}\label{222forG129} The graph $G129$ has a matrix with SMP and ordered multiplicity list $(2,2,2)$.  Furthermore, $q(G129)=q_M(G129)=3$.
\end{lemma}
\bpf  Observe that graph $G129$ can be constructed by adding a new vertex  adjacent to  two nonadjacent vertices $v$ and $w$ of $C_5=G38$.  In \cite[Theorem 48]{BFHHLS16} it was shown that 
$\wh C_5=\mtx{0 & 1 & 0 & 0 & -1 \\
 1 & 0 & 1 & 0 & 0 \\
 0 & 1 & 0 & 1 & 0 \\
 0 & 0 & 1 & 0 & 1 \\
 -1 & 0 & 0 & 1 & 0}$ has  SMP.  The eigenvalues of $\wh C_5$ are  $\mu_1=-2$, $\mu_2=\frac{1}{2} \left(1-\sqrt{5}\right)$, and $\mu_3=\frac{1}{2} \left(1+\sqrt{5}\right)$ with ordered multiplicity list  $(1,2,2)$.  Furthermore, the vector $[1, -1, 1, -1, 1]^T$ is a basis for the eigenspace of eigenvalue $\mu_1=-2$.  Thus it is not possible for an eigenvector $\bx$ for $\mu_1$ to have a zero entry, so   $|\supp(\bx)\cap \{v,w\}|= 2$.  Therefore, we can apply the Augmentation Lemma  to obtain a matrix having eigenvalue $\mu_1=-2$  with multiplicity $2$ and also eigenvalues $\mu_2$ and $\mu_3$ each with multiplicity 2.  Thus the graph $G129$ has a matrix with SMP and ordered multiplicity list $(2,2,2)$.  Since $G129$ has a unique shortest path on three vertices, $q(G129)=q_M(G129)=3$.
\epf
\vspace{-5pt}

\begin{remark}\label{222forG99}  Oblak and \v Smigoc   show that $G99$ has a matrix with every eigenvalue of even multiplicity \cite[Example 3.1]{OS14} and give a form to construct such a matrix in \cite[Theorem 3.1]{OS14}.  One such matrix is $M_{99}$ below.   They  also provided the matrix $M_{115}$  \cite{OSpersonal}, which they found in their research in preparation for \cite{OS14}.
\[M_{99}=\mtx{0 & 1 & 0 & 0 & 0 & 0 \\
 1 & 0 & 1 & 0 & 0 & -1 \\
 0 & 1 & 0 & 0 & 1 & 0 \\
 0 & 0 & 0 & 0 & 1 & 0 \\
 0 & 0 & 1 & 1 & 0 & 1 \\
 0 & -1 & 0 & 0 & 1 & 0}\!,\qquad  
 M_{115}=\mtx{
 -1 & 2 & 0 & 0 & 0 & 0 \\
 2 & -3 & -1 & -1 & 0 & 0 \\
 0 & -1 & 1 & 1 & 2 & 0 \\
 0 & -1 & 1 & 1 & -2 & 0 \\
 0 & 0 & 2 & -2 & -2 & \sqrt{3} \\
 0 & 0 & 0 & 0 & \sqrt{3} & 0}\!.\]  It is straightforward to verify that $\spec(M_{99})=\left\{-\sqrt{3},-\sqrt{3},0,0,\sqrt{3},\sqrt{3}\right\}$ and $\spec(M_{115})=\left\{-1-2 \sqrt{3},-1-2 \sqrt{3},0,0,-1+2 \sqrt{3},-1+2   \sqrt{3}\right\}$.  Since  each of $G99$ and $G115$ has a unique shortest path on three vertices,  $q(G99)=3=q(G115)$ by Theorem \ref{quniquedistance}.  Note that no matrix $A\in\sym(G99)$ or $A\in\sym(G115)$ that has $q(A)=3$ can have SMP because each is a spanning subgraph of $G134$, which has a unique shortest path on four vertices.
\end{remark}

Next we establish that $q(G)=2$ for various graphs $G$, starting with some that have SSP.
 The statement that  $q(G)$ can be realized by a matrix with SSP implies $q_M(G)=q(G)$, because SSP implies SMP.

\begin{lemma}\label{33byconstruct}  %The matrix 
Each matrix $M_{\#}$ below is orthogonal with SSP, $\oml(M_{G\#})=(3,3)$, and  $M_{\#}\in\sym(G\#)$ for the graphs $G\#=G174, G186$.  Thus $q(G\#)=2$. 

   \[\scriptsize M_{174}= \mtx{-\frac{1}{\sqrt{10}} & \sqrt{\frac{2}{5}} & \sqrt{\frac{2}{5}} &
   \frac{1}{\sqrt{10}} & 0 & 0 \\
 \sqrt{\frac{2}{5}} & -\frac{1}{\sqrt{10}} & \sqrt{\frac{2}{5}} & 0
   & \frac{1}{\sqrt{10}} & 0 \\
 \sqrt{\frac{2}{5}} & \sqrt{\frac{2}{5}} & -\frac{1}{\sqrt{10}} & 0
   & 0 & \frac{1}{\sqrt{10}} \\
 \frac{1}{\sqrt{10}} & 0 & 0 & \frac{1}{\sqrt{10}} &
   -\sqrt{\frac{2}{5}} & -\sqrt{\frac{2}{5}} \\
 0 & \frac{1}{\sqrt{10}} & 0 & -\sqrt{\frac{2}{5}} &
   \frac{1}{\sqrt{10}} & -\sqrt{\frac{2}{5}} \\
 0 & 0 & \frac{1}{\sqrt{10}} & -\sqrt{\frac{2}{5}} &
   -\sqrt{\frac{2}{5}} & \frac{1}{\sqrt{10}} }.\]  %end
 \[\scriptsize M_{186}= \mtx{\frac{1}{9} \left(-3-\sqrt{3}\right) & \frac{1}{18} \left(2
   \sqrt{3}-3\right) & 0 & \frac{1}{3}
   \sqrt{\frac{23}{6}-\frac{1}{\sqrt{3}}} & -\frac{1}{3} &
   -\frac{1}{2} \\
 \frac{1}{18} \left(2 \sqrt{3}-3\right) & \frac{1}{18} \left(-3-4
   \sqrt{3}\right) & \frac{1}{18} \left(3-2 \sqrt{3}\right) & 0 &
   \frac{2}{3} & -\frac{1}{2} \\
 0 & \frac{1}{18} \left(3-2 \sqrt{3}\right) & \frac{1}{9}
   \left(-3-\sqrt{3}\right) & \frac{1}{3}
   \sqrt{\frac{23}{6}-\frac{1}{\sqrt{3}}} & \frac{1}{3} &
   \frac{1}{2} \\
 \frac{1}{3} \sqrt{\frac{23}{6}-\frac{1}{\sqrt{3}}} & 0 &
   \frac{1}{3} \sqrt{\frac{23}{6}-\frac{1}{\sqrt{3}}} & \frac{1}{9}
   \left(3+\sqrt{3}\right) & 0 & 0 \\
 -\frac{1}{3} & \frac{2}{3} & \frac{1}{3} & 0 & \frac{1}{\sqrt{3}} &
   0 \\
 -\frac{1}{2} & -\frac{1}{2} & \frac{1}{2} & 0 & 0 & \frac{1}{2} }.\]
% is unitary with SSP and  $M_{186}\in\sym(G186)$ for the vertices of  $G186$ numbered as shown in Figure \ref{fig:graphs}. 
 \end{lemma}

%{\red [do we need the next figure?]}
%\begin{figure}[!ht] \begin{center}
%$K_3\cp P_2$ \qquad\scalebox{.6}{\includegraphics{G186}}\\ $G174$ \qquad\qquad\qquad\qquad\qquad\qquad  $G186$
%\caption{Graphs}\label{fig:graphsunit}\end{center}\end{figure}

\begin{corollary}\label{33fromG186}  %The matrix 
The graphs $G=G188, G192, G194, G196-G208$ have $q(G)=2$ and the ordered multiplicity list $(3,3)$ can be realized by a matrix with SSP.  \end{corollary}
\bpf Each of the graphs  $G188, G196, G198, G199, G202-G208$ has $G174$ as a spanning subgraph and each of  $G192, G194, G197, G200,$  $G201$ has $G186$ as a spanning subgraph.  So by Lemma \ref{33byconstruct} and Observation \ref{thm:sub}, $q(G)=2$.  
\epf

\begin{lemma}\label{33forG190G194} The graph $G190$ has a matrix with SSP and ordered multiplicity list $(3,3)$.
\end{lemma}

\bpf  The graph $G190$ is constructed by adding vertex 6 adjacent to $\{1,2,3\}$
of $G48$ (see  Figure \ref{fig:graphsaug}).  The ordered multiplicity list $(3,2)$ of $G48$ is realized by the matrix $M=\mtx{1 & 0 & \sqrt{2} & 1 & 1 \\
 0 & 1 & -\sqrt{2} & 1 & 1 \\
 \sqrt{2} & -\sqrt{2} & 4 & 0 & 0 \\
 1 & 1 & 0 & 2 & 2 \\
 1 & 1 & 0 & 2 & 2}$, which has SSP \cite[Lemma 3.5]{IEPG2}.  Furthermore, the vectors $[\frac{1}{2},\frac{1}{2},0,1,1]^T$ and $[\frac{1}{2 \sqrt{2}},-\frac{1}{2
   \sqrt{2}},1,0,0]^T$ are a basis for the eigenspace of eigenvalue 5.  Thus, it is not possible for an eigenvector for $\lam=5$ to have more than two zero entries, and the only way to achieve two zeros in an eigenvalue for $\lam=5$ is to have the zeros  in positions 4 and 5.  Therefore, $|\supp(\bx)\cap \{1,2,3\}|\ge 2$, and we can apply the Augmentation Lemma  to conclude 
     there is a matrix $B\in\sym(G190)$ which has SSP and has eigenvalues $\lam\!=\!5$ and $\lam\!=\!0$ each with multiplicity $3$.
\epf

\begin{corollary}\label{33fromG190G194}  %The matrix 
For $G=G195$, $q(G)=2$ and ordered multiplicity list $(3,3)$ can be realized by a matrix with SSP.  \end{corollary}
\bpf The graph $G195$ has $G190$ as a spanning subgraph, so by Lemma \ref{33forG190G194} and Observation \ref{thm:sub}, $q(G195)$ and ordered multiplicity list $(3,3)$ can be realized by a matrix with SSP.  
\epf

The next result can be verified by computation.

\begin{lemma}\label{33more}  %The matrix 
Each matrix $M_{\#}$ below is orthogonal, $\oml(M_{G\#})=(3,3)$, and  $M_{\#}\in\sym(G\#)$ for the graphs $G\#$. % shown in Figure \ref{fig:graphsunit}.  
Thus $q(G\#)=2$. 
\[M_{154}=\] \[\tiny \mtx{\frac{1}{2} & -\frac{1}{4} \sqrt{\frac{1}{2}
   \left(7-\sqrt{33}\right)} & 0 & -\frac{1}{2} & 0 & -\frac{1}{4}
   \sqrt{\frac{15-\sqrt{33}}{7-\sqrt{33}}} \\
 -\frac{1}{4} \sqrt{\frac{1}{2} \left(7-\sqrt{33}\right)} &
   -\frac{1}{2} & \frac{1}{2} \sqrt{\frac{9-\sqrt{33}}{7-\sqrt{33}}}
   & 0 & -\frac{1}{4} & 0 \\
 0 & \frac{1}{2} \sqrt{\frac{9-\sqrt{33}}{7-\sqrt{33}}} &
   \frac{1}{2} & -\frac{1}{4} \sqrt{\frac{1}{2}
   \left(9-\sqrt{33}\right)} & 0 & 0 \\
 -\frac{1}{2} & 0 & -\frac{1}{4} \sqrt{\frac{1}{2}
   \left(9-\sqrt{33}\right)} & -\frac{1}{2} & -\frac{1}{\sqrt{2
   \left(7-\sqrt{33}\right)}} & 0 \\
 0 & -\frac{1}{4} & 0 & -\frac{1}{\sqrt{2 \left(7-\sqrt{33}\right)}}
   & \frac{1}{2} & \frac{1}{4} \sqrt{\frac{1}{2}
   \left(15-\sqrt{33}\right)} \\
 -\frac{1}{4} \sqrt{\frac{15-\sqrt{33}}{7-\sqrt{33}}} & 0 & 0 & 0 &
   \frac{1}{4} \sqrt{\frac{1}{2} \left(15-\sqrt{33}\right)} &
   -\frac{1}{2}}\]  % end 
\[\scriptsize M_{168}= \mtx{-\frac{7}{12} & 0 & \frac{1}{12} & \frac{\sqrt{\frac{3}{2}}}{2} &
   -\frac{\sqrt{\frac{5}{6}}}{3} & -\frac{\sqrt{\frac{5}{3}}}{3} \\
 0 & -\frac{2}{3} & 0 & 0 & \frac{\sqrt{\frac{10}{3}}}{3} &
   -\frac{\sqrt{\frac{5}{3}}}{3} \\
 \frac{1}{12} & 0 & -\frac{7}{12} & \frac{\sqrt{\frac{3}{2}}}{2} &
   \frac{\sqrt{\frac{5}{6}}}{3} & \frac{\sqrt{\frac{5}{3}}}{3} \\
 \frac{\sqrt{\frac{3}{2}}}{2} & 0 & \frac{\sqrt{\frac{3}{2}}}{2} &
   \frac{1}{2} & 0 & 0 \\
 -\frac{\sqrt{\frac{5}{6}}}{3} & \frac{\sqrt{\frac{10}{3}}}{3} &
   \frac{\sqrt{\frac{5}{6}}}{3} & 0 & \frac{2}{3} & 0 \\
 -\frac{\sqrt{\frac{5}{3}}}{3} & -\frac{\sqrt{\frac{5}{3}}}{3} &
   \frac{\sqrt{\frac{5}{3}}}{3} & 0 & 0 & \frac{2}{3}}\]  % end 
\[\scriptsize M_{181}= \mtx{0 & \frac{1}{\sqrt{2}} & \frac{1}{\sqrt{2}} & 0 & 0 & 0 \\
 \frac{1}{\sqrt{2}} & -\frac{1}{4} & \frac{1}{4} & \frac{1}{8} & \frac{\sqrt{23}}{8} & 0 \\
 \frac{1}{\sqrt{2}} & \frac{1}{4} & -\frac{1}{4} & -\frac{1}{8} & -\frac{\sqrt{23}}{8} & 0 \\
 0 & \frac{1}{8} & -\frac{1}{8} & \frac{1}{8} \left(4-\sqrt{23}\right) & \frac{1}{8} & \frac{1}{4} \sqrt{\frac{11}{2}+2 \sqrt{23}} \\
 0 & \frac{\sqrt{23}}{8} & -\frac{\sqrt{23}}{8} & \frac{1}{8} & \frac{1}{2}-\frac{1}{8 \sqrt{23}} & -\frac{1}{4}
   \sqrt{\frac{11}{46}+\frac{2}{\sqrt{23}}} \\
 0 & 0 & 0 & \frac{1}{4} \sqrt{\frac{11}{2}+2 \sqrt{23}} & -\frac{1}{4} \sqrt{\frac{11}{46}+\frac{2}{\sqrt{23}}} &
   \frac{3}{\sqrt{23}}-\frac{1}{2} }\]  % end 
 \end{lemma}

For graphs $G=G154, G168, G181$ and matrix $A\in\sym(G)$, if $q(A)=2$, then $A$ does not have SMP, because in each case it is possible to add an edge to $G$ and obtain a unique shortest path on 3 vertices.

\begin{remark} As it may be useful for future research, here we briefly describe the method that was used to find the matrices  $M_{186}$ and $M_{168}$.  The graph $G186$ has three independent vertices, which we label 4, 5, and 6.  Vertices 1, 2, and 3 form a clique missing one edge.  All but one of the possible edges between vertices in $\{1,2,3\}$ and $\{4,5,6\}$ are present.  Thus we have the form $M_{186}=\mtx{A & C\\ C^T & D}$ where $D$ is diagonal, C has one  zero, $A^T=A$, and $A$ has one pair of symmetrically placed zeros.  In order for $M_{186}$ to be orthogonal, we must have $C^TC+D^2=I$, so the columns of $C$ are orthogonal (but may have different lengths).  Then $AC+CD=0$, so $A=-CDC^{-1}$.  The conditions \vspace{-5pt}
\begin{enumerate}[(i)]
\item $D$ is diagonal with distinct diagonal entries strictly between zero and one,  \vspace{-3pt}
\item the columns of $C$ are orthogonal and scaled so that $C^TC+D^2=I$, and \vspace{-3pt}
\item $A=-CDC^{-1}$ \vspace{-5pt}
\end{enumerate}
suffice to ensure $ \mtx{A & C\\ C^T & D}$ is orthogonal.    The columns of $C$ can be chosen with a zero in the first column, and one diagonal entry of $D$ can be used as a parameter that is set to achieve the desired pair of zeros in $A$.  The case of $M_{168}$ is similar except that now there are two pairs of zeros in $A$, and some care must be taken in the choice of the vectors for $C$.
\end{remark}
\vspace{-5pt}

Next we show the two graphs $G187$ and $G189$ have $q(G)=3$ by showing they do not allow an orthogonal realization.

\begin{lemma}\label{noq2g187} The graph $G187$, the wheel on 6 vertices,  does not allow an orthogonal matrix and  $q(G)=3$.
\end{lemma}  
\bpf  Since  $G197$ has a Hamilton cycle, $q(G187)\le 3$ by Corollary \ref{HamCycle}.  Showing that $G187$ does not allow an orthogonal matrix completes the proof because  $q(G)=2$ implies $G$ allows an orthogonal matrix by Observation \ref{q2}.     We have the following matrix:

\[M = 
\begin{bmatrix}
a & w & 0 & 0 & v & q \\
w & b & x & 0 & 0 & r \\
0 & x & c & y & 0 & s \\
0 & 0 & y & d & z & t \\
v & 0 & 0 & z & e & u \\
q & r & s & t & u & f
\end{bmatrix}\]

Suppose $M$ is orthogonal, so  $M^2=I$ where $M^2=$
{\tiny
\[
\left[\begin{array}{rrrrrr}
a^{2} + q^{2} + v^{2} + w^{2} & q r + a w + b w & q s + w x & q t + v z & q u + a v + e v & a q + f q + u v + r w \\
q r + a w + b w & b^{2} + r^{2} + w^{2} + x^{2} & r s + b x + cx & r t + x y & r u + v w & b r + f r + q w + s x \\
q s + w x & r s + b x + c x & c^{2} + s^{2} + x^{2} + y^{2} & s t + c y + d y & s u + y z & c s + f s + r x + t y \\
q t + v z & r t + x y & s t + c y + d y & d^{2} + t^{2} +
y^{2} + z^{2} & t u + d z + e z & d t + f t + s y + u z \\
q u + a v + e v & r u + v w & s u + y z & t u + d z + e z
& e^{2} + u^{2} + v^{2} + z^{2} & e u + f u + q v + t z \\
a q + f q + u v + r w & b r + f r + q w + s x & c s + f s + r x
+ t y & d t + f t + s y + u z & e u + f u + q v + t z &
f^{2} + q^{2} + r^{2} + s^{2} + t^{2} + u^{2}
\end{array}\right]\!\!.
\]}
We denote the $i,j$-entry of $M^2$ by $h_{ij}$,  we know $h_{ij}=0$ for $i\ne j$, and we apply this repeatedly to specific entries.  

{\footnotesize \begin{eqnarray} 
0=h_{1,3}=q s + w x &\Rightarrow& x=-\frac{qs}w\label{h13-x}\\
0=h_{2,5}=r u + v w &\Rightarrow& v=-\frac{ru}w\label{h25-v}\\
0=h_{3,5}= s u + y z &\Rightarrow& z=-\frac{su}y\label{h35-z}\\
0=h_{3,6}=q s + w x &\Rightarrow& y=-\frac{qs}w\label{h36-y}
\end{eqnarray}}
{\footnotesize \begin{eqnarray}
0=h_{15}=q u + a v + e v=0 &\Rightarrow& w=-\frac{(a+e)r}{q} \mbox{ and }a+e\ne 0\label{h15-w-1}\\
\eqref{h13-x} \mbox{ and }\eqref{h15-w-1} &\Rightarrow& x=-\frac{qs}{(a+e)r} \label{h13-x-2}\\
\eqref{h25-v} \mbox{ and }\eqref{h15-w-1} &\Rightarrow& v=-\frac{qu}{a+e} \label{h25-v-3}\\
\eqref{h13-x-2} \mbox{ and }\eqref{h36-y}= &\Rightarrow& y=\frac{s}{t} \left( \frac{q^2-(c+f)(a+e)}{a+e} \right) \label{h36-y-4}\\
\eqref{h36-y-4} \mbox{ and }\eqref{h35-z}= &\Rightarrow& z=\frac{-tu(a+e)}{q^2-(c+f)(a+e)} \label{h35-z-5}\\
\eqref{h15-w-1} \mbox{ and }0=h_{1,2}=q r + a w + b w &\Rightarrow& %\frac{r \left((a+b)(b+e)+q^2\right)}{q}=0 \Rightarrow 
q^2=-(a+b)(a+e) \label{h12-q2-6}\\
\eqref{h15-w-1}, \eqref{h12-q2-6}, \mbox{ and }\\0=h_{2,3}=r s + b x + cx &\Rightarrow& %\frac{s \left((a+b) (b+c)+r^2\right)}{r}=0 \Rightarrow 
r^2=-(a+b)(b+c) \label{h23-r2-7}\nonumber\\
\eqref{h36-y-4}, \eqref{h12-q2-6}, \mbox{ and }\\0=h_{3,4}=s t + c y + d y &\Rightarrow&  t^2=(c+d)(a+b+c+f) \label{h34-t2-8}\nonumber\\
\eqref{h35-z-5} \mbox{ and }0=h_{4,5}= t u + d z + e z &\Rightarrow&  q^2=(a+e)(c+d+e+f) \label{h45-q2-9}\\
\eqref{h12-q2-6}, \eqref{h45-q2-9},  \mbox{ and }\\a+e\ne 0 &\Rightarrow&  a+b+c+d+e+f=0 \label{eq6-eq9-10}\nonumber
\\%\eqref{h15-w-1}, \eqref{h12-q2-6}, \mbox{ and }0=h_{3,4}=s t + c y + d y &\Rightarrow&  t^2=(c+d)(a+b+c+f) \label{h34-t2-8}\\
\eqref{h15-w-1}, \eqref{h25-v-3}, \eqref{h12-q2-6}, \eqref{h23-r2-7}, \eqref{eq6-eq9-10}, \mbox{ and }\\0=h_{1,6}=a q + f q + u v + r w &\Rightarrow&  u^2=-(a+e)(d+e) \label{h16-u2-11}\nonumber\\
\eqref{h15-w-1}, \eqref{h13-x-2}, \eqref{h12-q2-6}, \eqref{h23-r2-7}, \eqref{eq6-eq9-10}, \mbox{ and }\\
0=h_{2,6}=b r + f r + q w + s x &\Rightarrow&  s^2=-(c+d)(b+c) \label{h26-s2-12}\nonumber
\end{eqnarray} }
We then consider the following chart, %Table \ref{tab:q3G187}, 
which begins with two possible cases for equation \eqref{h12-q2-6} using $q^2>0$.  Each of these cases is then applied successively to other equations that require positive values. %, deriving consequences that both lead to the contradiction that $(a+e)<0$ and  $(a+e)>0$:

\begin{center}
%\begin{table}
{\small \begin{tabular}{c|c|c|c}
\# & equation & Case 1 & Case 2\\
\hline
\eqref{h12-q2-6} & $q^2=-(a+e)(a+b) > 0$ & $(a+e)>0$  & $(a+e)<0 $    \\
& & and $(a+b) <0$ & and $ (a+b) >0$ \\\hline
\eqref{h23-r2-7} & $r^2=-(a+b)(b+c) > 0$ & $(b+c)>0$ & $(b+c)<0$\\
\hline
 \eqref{h26-s2-12} & $s^2=-(c+d)(b+c) > 0$ & $(c+d)<0$ & $(c+d)>0$\\ \hline
\eqref{h34-t2-8} & $t^2=-(d+e)(c+d) > 0$ & $(d+e)>0$ & $(d+e)<0$\\ \hline
\eqref{h16-u2-11}& $u^2=-(a+e)(d+e) > 0$ & $(a+e)<0$ & $(a+e)>0$
\end{tabular}}
%\caption{Case analysis for $G187$ \label{tab:q3G187}}
%\end{table}
\end{center}
In each case, we find the contradiction that $(a+e)<0$ and  $(a+e)>0$.
\epf

\begin{lemma}\label{noq2g189} Graph $G189$  does not allow an orthogonal matrix and  $q(G)=3$.
\end{lemma}  
\bpf
Since $G96$ is a subgraph of $G189$, $q(G189)\le 3$ by  Observation \ref{thm:sub}.  Showing that $G189$ does not allow an orthogonal matrix completes the proof because  $q(G)=2$ implies $G$ allows an orthogonal matrix by Observation \ref{q2}.     We have the following matrix:

\[M=\begin{bmatrix}
a & q & r & t & 0 & 0 \\
q & b & s & 0 & u & v \\
r & s & c & 0 & w & x \\
t & 0 & 0 & d & y & z \\
0 & u & w & y & e & 0 \\
0 & v & x & z & 0 & f
\end{bmatrix}\!\!.\]

Suppose $M$ is orthogonal, so $M^2$ is the identity matrix. Observe that $M^2$ is

{\tiny
\[\begin{bmatrix}
a^{2} + q^{2} + r^{2} + t^{2} & a q + b q + r s & a r + c r + q
s & a t + d t & q u + r w + t y & q v + r x + t z \\
a q + b q + r s & b^{2} + q^{2} + s^{2} + u^{2} + v^{2} & q r +
b s + c s + u w + v x & q t + u y + v z & b u + e u + s w &
b v + f v + s x \\
a r + c r + q s & q r + b s + c s + u w + v x & c^{2} + r^{2} +
s^{2} + w^{2} + x^{2} & r t + w y + x z & s u + c w + e w &
s v + c x + f x \\
a t + d t & q t + u y + v z & r t + w y + x z & d^{2} +
t^{2} + y^{2} + z^{2} & d y + e y & d z + f z \\
q u + r w + t y & b u + e u + s w & s u + c w + e w & d y +
e y & e^{2} + u^{2} + w^{2} + y^{2} & u v + w x + y z \\
q v + r x + t z & b v + f v + s x & s v + c x + f x & d z +
f z & u v + w x + y z & f^{2} + v^{2} + x^{2} + z^{2}
\end{bmatrix}\]}
and denote the $i,j$-entry of $M^2$ by $h_{ij}$.

Note  that $h_{1,4}=h_{4,5}=h_{4,6}=0$ implies  $a=-d=e=f$. We make these substitutions in $M$ and $M^2$ becomes
{\tiny
\[\begin{bmatrix}
a^{2} + q^{2} + r^{2} + t^{2} & a q + b q + r s & a r + c r + q
s & 0 & q u + r w + t y & q v + r x + t z \\
a q + b q + r s & b^{2} + q^{2} + s^{2} + u^{2} + v^{2} & q r +
b s + c s + u w + v x & q t + u y + v z & a u + b u + s w &
a v + b v + s x \\
a r + c r + q s & q r + b s + c s + u w + v x & c^{2} + r^{2} +
s^{2} + w^{2} + x^{2} & r t + w y + x z & s u + a w + c w &
s v + a x + c x \\
0 & q t + u y + v z & r t + w y + x z & a^{2} + t^{2} +
y^{2} + z^{2} & 0 & 0 \\
q u + r w + t y & a u + b u + s w & s u + a w + c w & 0
& a^{2} + u^{2} + w^{2} + y^{2} & u v + w x + y z \\
q v + r x + t z & a v + b v + s x & s v + a x + c x & 0
& u v + w x + y z & a^{2} + v^{2} + x^{2} + z^{2}
\end{bmatrix}\!\!.\]}
Denote the $i,j$-entry of this matrix by $k_{ij}$. We know $k_{ij}=0$ for $i\not= j$, and we apply this repeatedly to specific entries.

{\footnotesize \begin{eqnarray} 
0=k_{1,2}=a q + b q + r s \mbox{ and }  0=k_{1,3}=a r + c r + qs &\Rightarrow& q^2=\frac{(a+c)}{(a+b)} r^2\label{qr}\\
0=k_{1,2}=a q + b q + r s \mbox{ and }  0=k_{1,3}=a r + c r + qs &\Rightarrow& s^2=(a+b)(a+c)\label{s}\\
0=k_{2,6}=a v + b v + s x \mbox{ and }  0=k_{3,6}=s v + a x + c x &\Rightarrow& v^2=\frac{(a+c)}{(a+b)}x^2\label{vx}\\
\null \qquad  0=k_{2,5}=a u + b u + s w \mbox{ and }  0=k_{3,5}=s u + a w + c w &\Rightarrow& u^2=\frac{(a+c)}{(a+b)}w^2\label{uw}
\end{eqnarray} }

From $\eqref{qr}$ and $\eqref{s}$, $q=\pm \sqrt{\frac{(a+c)}{(a+b)}}r$ and $s=\pm\sqrt{(a+b)(a+c)}$.  If $q$ and $s$ both use positive roots or both negative roots,
\begin{eqnarray*} 
0=k_{1,2}=a q + b q + r s &\Rightarrow& 2\sqrt{(a+b)(a+c)}=0,
\end{eqnarray*} 
which is a contradiction. Therefore $q$ and $s$ must use roots of opposite sign. Similarly, we can see $v$ and $u$ must use roots of sign opposite to the sign of the root in the formula for  $s$ as well. Thus we have the following two cases.

\noindent {\bf Case 1:} For the first case, we let $s=\sqrt{(a+b)(a+c)}$, $q=-\sqrt{\frac{(a+c)}{(a+b)}} r$, $v=-\sqrt{\frac{(a+c)}{(a+b)}} x$, and $u=-\sqrt{\frac{(a+c)}{(a+b)}} w$. 

{\footnotesize \begin{eqnarray} 
0=k_{1,5}=q u + r w + t y, \\\eqref{qr}, \mbox{ and } \eqref{uw}  &\Rightarrow& y=\frac{-rw(2a+b+c)}{t(a+b)}\label{y}\nonumber\\
0=k_{1,6}=q v + r x + t z, \\ \eqref{qr}, \mbox{ and } \eqref{vx} &\Rightarrow& z=\frac{-rx(2a+b+c)}{t(a+b)}\label{z}\nonumber\\
0=k_{5,6}=u v + w x + y z, \\ \eqref{y}, \mbox{ and } \eqref{z} &\Rightarrow& t=\sqrt{-\frac{(2a+b+c)^2r^2}{(a+c)(a+b)+(a+b)^2}}\label{t}\nonumber
\end{eqnarray} }
Equation $\eqref{t}$ yields a contradiction since $t$ is imaginary.

\hspace{4pt}

\noindent {\bf Case 2:} For the second case, we let $s=-\sqrt{(a+b)(a+c)}$, $q=\sqrt{\frac{(a+c)}{(a+b)}} r$, $v=\sqrt{\frac{(a+c)}{(a+b)}} x$, and $u=\sqrt{\frac{(a+c)}{(a+b)}} w$. 

We observe that the same equations result from case 2 as in case 1 and we obtain the same contradiction.
\epf

Finally we establish the value of $q$ for the few remaining graphs.

\begin{remark}\label{q5}  It is well known that the path  is the only graph for which $q(G)=|V(G)|$ (see \cite[Proposition 3.1]{AACFMN13}).  Thus $q(G83=P_6)=6$.    It is shown in \cite{AACFMN13} that  $G=S(k-1,n-k-1,1)$  and $G=GB(k,n-k-3)$ have $q(G)=n-1$.  Since $G80=S(2,2,1)$ and $G81=S(3,1,1)$, $q(G80)=5=q(G81)$.  Since $G97=GB(2,1)$ and $G102=GB(3,0)$,   $q(G97)=5=q(G102)$.   By Theorem  \ref{thm:qisGminus1}, every graph other than $G80, G81, G83, G97, G102$ has $q(G)\le 4$.  For $G = G78$, $G79$, $G93$, $ G94$, $ G95$, $ G98$, $ G100$, $ G103$, $ G104$, $ G112$, $ G113$, $ G119$, $ G120$, $ G122-G124$, 
%$ G113$, $ G119$, $  G120$, $  G122, G123,  $, 
$G130, G134, G139, G142$, $G$  has a unique shortest path on 4 vertices, so $q(G)=4$.
\end{remark}

We have now established $q$ for all graphs of order six.  For each graph, a reason is  given. 
\begin{theorem}\label{ord6q}  Tables \ref{tab:ord6q} lists the value of $q$ for each connected graph of order six. \end{theorem}

\newpage
%\vspace{-20pt}
\begin{table}[h!]
\caption{Values of $q$ for graphs of order 6 using the graph numbering in \cite{atlas}.  A column headed \# gives the result \# that justifies the corresponding $q(G\#)$. \label{tab:ord6q}}\vspace{-12pt}
\begin{center}
{\small \begin{tabular}{|c|c|c||c|c|c||c|c|c|}
\hline
$G\#$ & $q(G\#)$ &  \# & $G\#$ & $q(G\#)$ &  \# & $G\#$ & $q(G\#)$ &  \# \\
\hline\hline
$G77$ &	3	&  \ref{cliquestar}  &
$G78$ &	4	&	\ref{q5}
&					
$G79$ &	4	&	\ref{q5}
\\
\hline					
$G80$ &	5	&	\ref{q5} 
&					
$G81$ &	5	&	\ref{q5} 
&
$G83$ &	6	&	\ref{q5} 
\\
\hline
$G92$ &	3	&	\ref{cliquestar}
&
$G93$ &	4	&	\ref{q5}
&
$G94$ &	4	&	\ref{q5}
\\
\hline
$G95$ &	4	&	\ref{q5}
&
$G96$ &	3	&	\ref{222forG96}
&
$G97$ &	5	&	\ref{q5} 
\\
\hline
$G98$ &	4	&	\ref{q5}
&
$G99$ &	3	&	\ref{222forG99}
&
$G100$ &	4	&  \ref{q5}
\\
\hline
$G102$ &	5	&	\ref{q5} 
&
$G103$ &	4	&	\ref{q5}
&
$G104$ &	4	&    \ref{q5}
\\
\hline
$G105$ &	3	&	\ref{222forG105}
&
$G111$ &	3	&	\ref{222fromG96}
&
$G112$ &	4	& \ref{q5}
\\
\hline
$G113$ & 4	&	\ref{q5}
&
$G114$ &	3	&	\ref{222fromG96}
&
{$G115$} &	{3}	& \ref{222forG99}
\\
\hline
$G117$ &	3	&	\ref{cliquestar}
&
$G118$ &	3	&	\ref{222fromG96}
&
$G119$ &	4	& \ref{q5}
\\
\hline
$G120$ &	4	&	\ref{q5} 
&
$G121$ &	3	&	\ref{222fromG96}
&
$G122$ &	4	& \ref{q5} 
\\
\hline
$G123$ &	4	&	\ref{q5} 
&
$G124$ &	4	&	\ref{q5} 
&
$G125$ &	3	&  \ref{222forG125}
\\
\hline
$G126$ &	3	&		\ref{222fromG96}
&
$G127$ &	3	&	\ref{222forG105}
&
$G128$ &	3	&  \ref{PsboxP2}
\\
\hline
$G129$ &	3	&	\ref{222forG129}	
&
$G130$ &	4	&	\ref{q5} 
&
$G133$ &	3	&  \ref{222fromG96}
\\
\hline
$G134$ &	4	&	\ref{q5} 	
&
$G135$ &	3	&	\ref{222fromG96}
&
$G136$ &	3	&  \ref{222fromG96}
\\
\hline
$G137$ &	3	&		\ref{222fromG96}
&
$G138$ &	3	&	\ref{222fromG125}
&
$G139$ &	4	&  \ref{q5}
\\
\hline
$G140$ &	3	&		\ref{222fromG96}
&
$G141$ &	3	&	\ref{222fromG96}
&
$G142$ &	4	&  \ref{q5}
\\
\hline
$G143$ &	3	&		\ref{222fromG125}
&
$G144$ &	3	&	\ref{222fromG96}
&
$G145$ &	3	&  \ref{222fromG96}
\\ \hline
$G146$ &	3	&	\ref{qpaperKmn}	
&
$G147$ &	3	&	\ref{222forG105}
&
$G148$ &	3	&  \ref{222forG105}
\\
\hline
$G149$ &	3	&		\ref{222fromG96}
&
$G150$ &	3	&	\ref{222fromG96}
&
$G151$ &	3	&     \ref{222forG105}
\\
\hline
$G152$ &	3	&	 \ref{222forG105}
&
$G153$ &	3	&	 \ref{222forG105}
&
$G154$ &	2	&  \ref{33more}
\\
\hline
$G156$ &	3	&		\ref{222fromG96}
&
$G157$ &	3	&	\ref{222fromG96}
&
$G158$ &	3	&  \ref{222fromG96}
\\
\hline
$G159$ &	3	&		\ref{222fromG96}
&
$G160$ &	3	&	\ref{222fromG125}
&
$G161$ &	3	&  \ref{222fromG96}
\\
\hline
$G162$ &	3	&		\ref{222fromG96}
&
$G163$ &	3	&	\ref{222fromG96}
&
$G164$ &	3	&  \ref{222fromG96}
\\
\hline
$G165$ &	3	&		\ref{222fromG96}
&
$G166$ &	3	&	\ref{222fromG96}
&
$G167$ &	3	&  \ref{222fromG96}
\\
\hline
$G168$ &	2	&	\ref{33more}	
&
$G169$ &	3	&	\ref{222fromG96}
&
$G170$ &	3	&  \ref{222fromG96}
\\
\hline
$G171$ &	3	&		\ref{222fromG96}
&
$G172$ &	3	&	\ref{222fromG96}
&
$G173$ &	3	&  \ref{222fromG96}
\\
\hline
$G174$ &	2	&		\ref{33byconstruct}
&
$G175$ &	2	&	\ref{qpaperKmn}
&
$G177$ &	3	&  \ref{222fromG96}
\\
\hline
$G178$ &	3	&	\ref{222fromG96}
&
$G179$ &	3	&	\ref{222fromG96}
&
$G180$ &	3	&  \ref{222fromG96}
\\
\hline
$G181$ &	 {2}	&	\ref{33more} &
$G182$ &	3	&	\ref{222fromG96}
&
$G183$ &	3	&  \ref{222fromG96}
\\
\hline
$G184$ &	3	&	\ref{222fromG96}
&
$G185$ &	3	&	\ref{222fromG96}
&
$G186$ &	2	& \ref{33byconstruct}
\\
\hline
$G187$ &	3	&	\ref{noq2g187}
&
$G188$ &	2	&	\ref{33fromG186}
&
$G189$ &	3	&  \ref{noq2g189}
\\
\hline
$G190$ &	2	&	\ref{33forG190G194}
&
$G191$ &	3	&	\ref{cliquepath}
&
$G192$ &	2	& \ref{33fromG186}
\\
\hline
$G193$ &	3	&	\ref{222fromG96}
&
$G194$ &	2	&	\ref{33fromG186}
&
$G195$ &	2	& \ref{33fromG190G194}
\\
\hline
$G196$ &	2	&	\ref{33fromG186}
&
$G197$ &	2	&	\ref{33fromG186}
&
$G198$ &	2	& \ref{33fromG186}
\\
\hline
$G199$ &	2	&	\ref{33fromG186}
&
$G200$ &	2	&	\ref{33fromG186}
&
$G201$ &	2	& \ref{33fromG186}
\\
\hline
$G202$ &	2	&	\ref{33fromG186}
&
$G203$ &	2	&	\ref{33fromG186}
&
$G204$ &	2	& \ref{33fromG186}
\\
\hline
$G205$ &	2	&	\ref{33fromG186}
&
$G206$ &	2	&	\ref{33fromG186}
&
$G207$ &	2	& \ref{33fromG186}
\\
\hline
$G208$ &	2	&	\ref{33fromG186}
&
 &		&	
&
 &		& 
 \\
\hline
\end{tabular}}
\end{center}
\end{table}

\newpage
%%%%%%%%%%%%%%%%%%%%%%%%%%%%%%%%%%%%%%% 
\section{Values of $q$ for families of graphs}\label{sqfam} 

The next table summarizes  known values of $q(G)$.
\begin{center}
%\vspace{-20pt}
\begin{table}[ht!]
\caption{Values of $q$ for families of graphs\label{tab:q-families}}
\renewcommand{\arraystretch}{1.1}
\begin{tabular}{|c|c|l|}
\hline
Graph $G$	& $q(G)$ & Reason\\
\hline
$K_n$		&	2		&\cite[Lemma 2.2]{AACFMN13}\\ \hline
$C_n$		&	$\lceil \frac{n}{2} \rceil$ &\cite[Lemma 2.7]{AACFMN13}\\[2pt] \hline
$P_n$		&	$n$		&\cite[Proposition 3.1]{AACFMN13}\\ \hline
$K_{n,m}$	& $\left\{\begin{array}{cc} 2, & \text{if } m=n\\ 3,& \text{if } m<n \end{array}\right.$	 &\cite[Corollary 6.5]{AACFMN13}\\ \hline
$Q_n$		& 2			&\cite[Corollary 6.9]{AACFMN13} \\ \hline
$GB(k,n-k-3)$ & $n-1$	&\cite[Proposition 7.1]{AACFMN13}\\ \hline
$S(k-1,n-k-1,1)$	& $n-1$	&\cite[Proposition 7.2]{AACFMN13} \\ \hline
%$C_5'$ and $C_5''$	& 3	&\cite[Lemma 7.5]{AACFMN13} \\ \hline%\hline
$|V(G)|\le 6$ & & Tables \ref{tab:ord5q} and \ref{tab:ord6q}\\ \hline
$KP(n_1,n_2,\ldots,n_s)$ for $s\ge 2, n_i\geq 2$& $s+1$ & Theorem \ref{cliquepath}\\ \hline
  $KS(n_1,n_2,\ldots,n_s)$ for $s\ge 2, n_i\geq 2$&3 & Theorem \ref{cliquestar}\\ \hline
$P_s\cp P_2$ & $s$ & Corollary \ref{PsboxP2}\\ \hline
$C_4\cp P_{2s}$ & $2s$ & Corollary \ref{P2sboxC4}\\ 
\hline
$C_{4}\cp C_{s} \mbox{ for }s\ge 4 \ \& \ s\not\equiv 2\mod 4$ & $\lc\frac s 2\rc$ & Corollary \ref{P2sboxC4}\\[2pt] \hline
$P_s\times P_2$ &$s$ & Corollary \ref{PstimesP2}\\ \hline
$C_4\times P_s$ &$s$ & Proposition \ref{GXC4}\\ \hline
$P_3 \boxtimes P_3$ &$3$ &Corollary \ref{P3boxP3}\\ \hline
$P_s \vee K_1$ &$\lc\frac{s+1}2\rc$ &Example \ref{PsveeK1}\\[2pt] \hline
$K_n-e$ &$2$ &Corollary \ref{Kn-e}\\ \hline
%$P_{2k} \boxtimes P_2$ &$2k$ &Corollary \ref{prop:spp}\\ \hline
\end{tabular}
\end{table}
\end{center}

%\textcolor{red}{*no better notation in Ahmadi paper}

% The following can be obtained by a direct application of Theorem \ref{quniquedistance}
% \begin{center}
% \begin{tabular}{c|c}
% Graph	& $q$ \\
% \hline
% 	& \\
% \end{tabular}
% \end{center}

%%%%%%%%%%%%%%%%%%%%%%%%%%%%%%%%%%%%%%%% 

%%%%%%%%%%%%%%%%%%%%%%%%%%%%%%%%

\subsection*{Acknowledgements}  We thank Yu (John) Chan for discussions  about $q$ for graph products, Audrey Goodnight for providing
the matrix  $M$ for the Banner  used in the proof of   Lemma \ref{222forG125},  Polona Oblak and Helena \v Smigoc for providing the matrix $M_{115}$ in Remark \ref{222forG99}, and the anonymous referee for helpful comments that improved the exposition.

%%%%%%%%%%%%%%%%%%%%%%%%%%%%%%%%%%%%%%%% Bibliography

\end{document}